\documentclass[12pt,reqno]{amsart}
\usepackage{amsmath}
\usepackage{amssymb}
\usepackage{amstext}
\usepackage{a4wide}
\usepackage{graphicx}
\allowdisplaybreaks \numberwithin{equation}{section}
\usepackage{color}
\usepackage{cases}

\numberwithin{equation}{section}

\newtheorem{theorem}{Theorem}[section]

\newtheorem{corollary}[theorem]{Corollary}
\newtheorem{lemma}[theorem]{Lemma}
\newtheorem*{theoremA}{Theorem A}
\newtheorem*{theoremB}{Theorem B}
\newtheorem*{theoremC}{Theorem C}
\theoremstyle{definition}

\newtheorem{definition}[theorem]{Definition}

\theoremstyle{remark}
\newtheorem{remark}[theorem]{Remark}

\begin{document}

\title
[Desingularization of vortices for 2D steady Euler flows]{Desingularization of vortices for 2D steady Euler flows via the vorticity method}

 \author{Daomin Cao, Guodong Wang,  Weicheng Zhan}
\address{Institute of Applied Mathematics, Chinese Academy of Sciences, Beijing 100190, and University of Chinese Academy of Sciences, Beijing 100049,  P.R. China}
\email{dmcao@amt.ac.cn}
\address{Institute for Advanced Study in Mathematics, Harbin Institute of Technology, Harbin 150001, P.R. China}
\email{wangguodong14@mails.ucas.edu.cn}
\address{Institute of Applied Mathematics, Chinese Academy of Sciences, Beijing 100190, and University of Chinese Academy of Sciences, Beijing 100049,  P.R. China}
\email{zhanweicheng16@mails.ucas.ac.cn}


\begin{abstract}
In this paper, we consider steady Euler flows in a planar bounded domain in which the vorticity is sharply concentrated in a finite number of disjoint regions of small diameter. Such flows are closely related to the point vortex model and can be regarded as desingularization of point vortices. By an adaption of the vorticity method, we construct a family of steady Euler flows in which the vorticity is concentrated near a global minimum point of the Robin function of the domain,
 and the corresponding stream function satisfies a semilinear elliptic equation with a given profile function.
 Furthermore, for any given isolated minimum point $(\bar{x}_1,\cdot\cdot\cdot,\bar{x}_k)$ of the Kirchhoff-Routh function of the domain, we prove that there exists a family of steady Euler flows whose vorticity is supported in $k$ small regions near $\bar{x}_i$, and near each $\bar{x}_i$ the corresponding stream function satisfies a semilinear elliptic equation with a given profile function.
\end{abstract}

\maketitle

\section{Introduction and main results}

In this paper, we shall consider an incompressible inviscid fluid in two dimensions whose evolution is governed by the following Euler system
\begin{equation}\label{1-1}
\begin{cases}
\partial_t\mathbf{v}+(\mathbf{v}\cdot\nabla)\mathbf{v}=-\nabla P,&(x,t)\in D\times(0,+\infty),\\
\nabla\cdot\mathbf{v}=0,\\
\mathbf{v}\cdot\mathbf{n}=0,&(x,t)\in\partial D\times(0,+\infty),\\
\mathbf{v}(\cdot,0)=\mathbf{v}_0,&x\in D.
 \end{cases}
\end{equation}
Here $D\subset\mathbb R^2$ is a bounded and simply-connected domain with smooth boundary, $\mathbf{v}=(v^1,v^2)$ is the velocity field, $P$ is the scalar pressure and $\mathbf{n}$ is the outer unit normal to $\partial D$. The boundary condition $\mathbf{v}\cdot\mathbf{n}=0$ means that there is no matter flow through $\partial D$.

For a planar flow, the scalar vorticity is defined as the third component of the $curl$ of the velocity field, that is,
\[\omega:=\partial_{x^1}v^2-\partial_{x^2}v^1.\]
The evolution of vorticity is described by the following nonlinear transport equation
\begin{equation}\label{vor}
\partial_t\omega+\mathbf{v}\cdot\nabla\omega=0,\,\,(x,t)\in D\times(0,+\infty),
\end{equation}
which is usually called the vorticity equation. Besides, the velocity field can be recovered from the vorticity via the Biot-Savart law
\[\mathbf{v}=(\partial_{x^2}\mathcal{G}\omega,-\partial_{x^1}\mathcal{G}\omega),\,\,\mathcal{G}\omega(x)=\int_DG(x,y)\omega(y)dy,\]
where $G$ is the Green's function of $-\Delta$ in $D$ with zero Dirichlet data, which can be written as follows
\[G(x,y)=-\frac{1}{2\pi}\ln|x-y|-h(x,y),\,\,x,y\in D.\]
 In the sequel we will use $\mathbf {b}^\perp$ to denote the clockwise rotation through $\pi/2$ of some planar vector $\mathbf b$, and for some function $f$ we denote $\nabla^\perp f=(\nabla f)^\perp$ for simplicity. Therefore the vorticity equation can be written as follows
\begin{equation}\label{vorequ}
\partial_t\omega+\nabla^\perp\mathcal{G}\omega=0,\,\,(x,t)\in D\times(0,+\infty).
\end{equation}

If the initial vorticity is smooth, the global existence and uniqueness of smooth solutions of the vorticity equation go back to the work of H\"older \cite{Ho} and Wolibner \cite{Wo} in 1930s. For initial vorticity only in $L^\infty$, the global existence and uniqueness of weak solution are proved by Yudovich \cite{Y} in 1963.
To summarize Yudovich's result let us introduce some definitions first. Define the rearrangement class of $\theta$ by
 \begin{equation}
 \mathcal{R}({\theta}):=\{ v \in L^1_{loc}(D)\mid  |\{x\in D\mid v(x)>s\}|=|\{x\in D\mid \theta(x)>s\}| ,\forall s\in \mathbb R\},
 \end{equation}
 where $|\cdot|$ denotes the two-dimensional Lebesgue measure.

 The kinetic energy of the fluid is given by
 \[
 E(t):=\frac{1}{2}\int_D\int_DG(x,y)\omega(x,t)\omega(y,t)dxdy.
 \]

 The result of Yudovich \cite{Y} can be stated as follows.

 \begin{theoremA}\label{A}
 Let $\omega_0\in L^\infty(D)$. Then there exists a unique weak solution to the vorticity equation $\omega(x,t)\in L^\infty(D\times(0,+\infty))\cap C( [0,+\infty); L^p(D))$ for all $p\in[1,+\infty)$ satisfying
\begin{equation}\label{tweak}
  \int_D\omega_0(x)\xi(x,0)dx+\int_0^{+\infty}\int_D\omega(\partial_t\xi+\nabla\xi\cdot \nabla^\perp \mathcal G\omega)dxdt=0
  \end{equation}
for all $\xi\in C_c^{\infty}(D\times[0,+\infty))$.
 Moreover, this weak solution satisfies
\begin{itemize}
\item[(i)] $\omega(x,t)\in \mathcal{R}({\omega_0})$ for all $t\geq 0$;
\item[(ii)]  the kinetic energy of the fluid is conserved, that is,
\[E(t)=E(0),\,\,\forall\, t\in[0,+\infty).\]
\end{itemize}
 \end{theoremA}
 The detailed proof of Theorem A can be found in Burton \cite{B5} or Majda-Bertozzi \cite{MB}.

 Although the well-posedness of the 2D Euler equation has already been solved, there are many other interesting open problems, especially those in vortex dynamics, that are both challenging in mathematics and meaningful in physics. In this paper, we will be concerned with one of them, that is, the possible equilibria of steady Euler flows with concentrated vorticity.

Many natural phenomena exhibit a strong eddylike motion in a finite number of small regions while being irrotational elsewhere. To deal with such a problem mathematically, we need to consider the Euler evolution of sufficiently concentrated vorticity.
 To simplify the problem, we first assume that the vorticity is a delta measure(called a point vortex) at $x(t)$ with unit strength, that is, $\omega(\cdot,t)=\delta(x(t))$. Then at time $t$ the velocity field induced by this point vortex is
\[\mathbf{v}(\cdot,t)=\nabla_x^\perp G(x,x(t))=-\frac{1}{2\pi}\frac{(\cdot-x(t))^\perp}{|\cdot-x(t)|^2}-\nabla_x^\perp h(\cdot,x(t)),\]
where $h$ is the regular part of the Green's function.
Intuitively by symmetry the term $-\frac{1}{2\pi}\frac{(\cdot-x(t))^\perp}{|\cdot-x(t)|^2}$ does not influence the motion of the fluid particle at $x(t)$. If we drop this term, we get the equation of $x(t)$
\begin{equation*}
\frac{dx(t)}{dt}=-\nabla_x^\perp h(\cdot,x(t)).
\end{equation*}
 Similarly, if the vorticity is a sum of $k$ delta measures at $x_1(t),\cdot\cdot\cdot,x_k(t)$ with strength $\kappa_1,\cdot\cdot\cdot,\kappa_k$, then the evolution of $x_i(t)$ is described by the following system
\begin{equation}\label{kr}
\frac{dx_i(t)}{dt}=-\kappa_i\nabla_x^\perp h(x_i(t),x_i(t))+ \sum_{j=1,j\neq i}^k\kappa_j\nabla_x^\perp G(x_j(t),x_i(t)),\;i=1,\cdot\cdot\cdot,k.
\end{equation}
System \eqref{kr} is called the point vortex model or the Kirchhoff-Routh model, which can be regarded as the singular limit of the Euler equation when the vorticity is concentrated.
It is a Hamiltonian system with the following Kirchhoff-Routh function as the Hamiltonian
\begin{equation}\label{krf}
{\mathcal W}_k(x_1,\cdot\cdot\cdot,x_k):=- \sum_{i\neq j, 1\leq i,j\leq k}\kappa_i\kappa_jG(x_i,x_j)+ \sum_{i=1}^k\kappa_i^2h(x_i,x_i),
\end{equation}
where $x_i\in D$ and $x_i\neq x_j$ if $i\neq j$.
Note that for a single vortex with unit strength(that is, $k=1$ and $\kappa=1$), the Kirchhoff-Routh function reduces to the Robin function $H(x):=h(x,x)$. We refer the interested readers to Lin \cite{L} or Marchioro--Pulvirenti \cite{MPu} for a detailed discussion.
Note that when we deduce the point vortex model from the Euler equation, we drop the self-interaction for each point vortex, which is just not rigorous. A natural question is whether we can give the mathematical justification of the point vortex model. More precisely, if the initial vorticity is concentrated near $k$ different points $x_1(0),\cdot\cdot\cdot,x_k(0)$, we ask whether the evolved vorticity remains concentrated near $k$ points $x_1(t),\cdot\cdot\cdot,x_k(t)$, and whether these $k$ points satisfy the point vortex model. Such a problem is called desingularization of point vortices. By now there are many results in the literature dealing with the problem. See Marchiror \cite{M}, Marchioro--Pulvirenti \cite{MP1,MP2}, Turkington \cite{T3} for example.

Another parallel problem is the desingularization of steady state of the point vortex model, which is exactly what we will be focusing on in this paper. More precisely, for any given equilibrium state of the point vortex model, or equivalently a critical point of the Kirchhoff-Routh function, say $(\bar{x}_1,\cdot\cdot\cdot,\bar{x}_k)$, we aim to construct a family of steady solutions of the Euler equation such that the support of the vorticity is supported in $k$ small regions near $\bar{x}_i$ with circulation $\kappa_i$ and ``shrinks" to $\bar{x}_i$ as the parameter changes.

For a steady Euler flow, the vorticity satisfies the following equation
\begin{align}\label{sv}
  \nabla^\perp \mathcal{G}\omega\cdot\nabla\omega=0,\,\,x\in D,
\end{align}
which means that $\nabla\omega$ and $\nabla\mathcal{G}\omega$ are collinear at each point. For $\omega\in L^\infty(D)$, by \eqref{tweak} we have the following definition of weak solution to \eqref{sv}.
\begin{definition}\label{weak}
Let $\omega\in L^{\infty}(D)$. Then $\omega$ is called a weak solution to \eqref{sv} if it satisfies
\begin{align}\label{int}
  \int_D\omega(x)\nabla^\perp\mathcal{G}\omega(x)\cdot\nabla \phi(x)dx=0,\,\,\forall \phi\in C_c^\infty(D).
\end{align}
\end{definition}
Note that for $\omega\in L^{\infty}(D)$, by elliptic regularity theory and Sobolev embedding we have $\mathcal{G}\omega\in C^{1}(\overline{D})$, therefore the integral in \eqref{int} makes sense.

In the past decades, many efforts have been devoted to establishing possible equilibria of Euler flows. Roughly speaking, there are mainly two methods to deal with this problem. The first one is the vorticity method, which was first established by Arnold \cite{A}(see also Arnold--Khesin \cite{AK}) and later developed by many authors. See for example Badiani \cite{Ba}, Burton \cite{B2,B3}, Elcrat--Miller \cite{EM}, Eydeland--Turkington \cite{ET} and Turkington \cite{T,T2}. To explain the vorticity method, we begin with a brief description of Turkington's method in \cite{T}, where steady vortex patch solutions of desingularization type were constructed. Based on Arnold's idea, Turkington considered maximization of the kinetic energy $E(\omega)$ over the admissible class
\[\mathcal{M}^\lambda:=\{\omega\in L^\infty(D)\mid 0\leq\omega\leq \lambda \mbox{ a.e. in }D,\int_D\omega(x)dx=1\}.\]
Here $\lambda$ is a large positive number. Turkington proved that $E$ attains its maximum over $\mathcal{M}^\lambda$ and each maximizer $\omega^\lambda$ must be a steady solution to the vorticity equation with the form
\[\omega^\lambda=\lambda \chi_{_{\{x\in D\mid\mathcal{G}\omega(x)>\mu^\lambda\}}},\]
where $\chi$ denotes the characteristic function and $\mu^\lambda$ is the Lagrange multiplier depending on $\lambda$. Moreover, as $\lambda$ goes to infinity, the support of $\omega^\lambda$ ``shrinks" to a global minimum point of the Robin function of the domain. Later Burton \cite{B2,B3} generalized Turkington's result by replacing $\mathcal{M}^\lambda$ by a more general admissible class. More precisely, Burton proved that the kinetic energy $E$ attains its maximum value on any rearrangement class of a given $L^p$ function, and any maximizer must be a steady solution of the vorticity equation with the form
\begin{align}\label{ome}
  \omega=f(\mathcal{G}\omega),
\end{align}
where the profile function $f$ is an unknown nondecreasing function. As an application of Burton's theory, Elcrat--Miller \cite{EM} proved existence of steady Euler flows with vorticity concentrated in a finite number of small regions, and in each small region the vorticity also satisfies \eqref{ome}
for some unknown nondecreasing function $f$.

The vorticity method is a very efficient way to construct steady Euler flows. However, the fact that the profile function $f$ is unknown is somewhat annoying. In many problems, we need to know what $f$ is to give a better description of the steady flow, such as nonlinear stability. By Burton \cite{B5}, if we are able to prove that the maximizer is isolated over the rearrangement class, then the flow must be nonlinearly stable. However, isolatedness of the maximizer from the viewpoint of vorticity is usually hard to verify except for several special cases(for the case $D$ is a disc and the maximizer is a circular patch, isolatedness is proved in \cite{WP}). If we know what $f$ is, we can achieve this by analyzing uniqueness of solution to the elliptic equation satisfied by the stream function. For vortex patch solutions, local uniqueness of the corresponding elliptic problem was proved by Cao--Guo--Peng--Yan \cite{CGPY}, and then was used to prove nonlinear stability of concentrated steady vortex patches by Cao--Wang \cite{CW}.

Another way to construct steady Euler flows is to solve directly the following semilinear elliptic problem with Dirichlet condition for the stream function, which is usually called the stream function method
\begin{align}\label{semilinear}
\begin{cases}
-\Delta\psi=f(\psi),&x\in D,\\
\psi=0,&x\in\partial D.
\end{cases}
\end{align}
It is easy to check that if $f$ is locally Lipschitz, then $\mathbf{v}=\nabla^\perp\psi$ is a steady solution of the Euler equation with $P=\int_0^{\psi}f(r)dr-\frac{1}{2}|\nabla\psi|^2$. More generally, we have
\begin{theoremB}[Cao--Wang, \cite{CW7}]
Let $k$ be a positive integer.
Suppose that $\omega\in L^{\infty}(D)$ satisfies
\begin{equation}\label{form}
\omega=\sum_{i=1}^k\omega_i, \,\,\min_{1\leq i< j\leq k}\{dist(supp(\omega_i),supp(\omega_j))\}>0,\,\,\omega_i=f_i(G\omega) \text{ a.e. in } supp(\omega_i)_\delta
\end{equation}
for some $\delta>0$, where $supp(\cdot)$ denotes the essential support of some measurable function and
 \[supp(\omega_i)_\delta=\{x\in D\mid \mbox{dist}(x,supp(\omega_i))<\delta\},\]
 and each $f_i:\mathbb R\to\mathbb R$ is either monotone or locally Lipschitz continuous,
then $\omega$ is a weak solution to the steady vorticity equation \eqref{sv}.
\end{theoremB}
Here the definition of the essential support of a measurable function can be found in \S 1.5 in \cite{LL}.

Now we recall several results of desingularization that was based on the stream function method. In \cite{SV}, Smets--Schaftingen obtained steady Euler flows of the form \eqref{form} with $l=1$ and a $p$-power ($p>1$) nonlinearity by solving a constraint minimization problem for the stream function. Moreover, the support of the vorticity is concentrated near a minimum point of the Robin function. In \cite{CLW}, based on the reduction method, Cao--Liu--Wei generalized Smets--Schaftingen's result to general positive integer $l$ with the support of the vorticity concentrated near a given non-degenerate critical point of the Kirchhoff-Routh function.
In \cite{CPY}, still based on the reduction method, Cao--Peng--Yan constructed steady multiple vortex patch solutions (i.e., each $f_i$ is a Heaviside type function) with concentrated vorticity. For general $l$ and $p$-power nonlinearity with $p\in(0,1)$, the corresponding desingularization result was obtained by Cao--Peng--Yan in \cite{CPY2}.

The advantage of the stream function is that one can get more delicate estimate for the solutions. However, it is hard to characterize the energy level on rearrangement class from the viewpoint of vorticity which is essentially important to prove nonlinear stability.
For example, in \cite{SV} Smets--Schaftingen proved existence of the following elliptic problem for small $\varepsilon>0$
\begin{equation}\label{sve}
\begin{cases}
-\Delta u^\varepsilon=\frac{1}{\varepsilon^2}(u^\varepsilon-\mu^\varepsilon)_+^p,&x\in D,\\
u^\varepsilon=0,&x\in\partial D,\\
\int_D\frac{1}{\varepsilon^2}(u^\varepsilon-\mu^\varepsilon)_+^pdx=1+o(1),\\
supp((u^\varepsilon-\mu^\varepsilon)_+)\subset B_{o(1)}(\bar{x}).
\end{cases}
\end{equation}
where $1<p<+\infty$, $\mu^\varepsilon$ is a real number depending on $\varepsilon$, $\bar{x}$ is a global minimum point of the Robin function, and $o(1)\to0$ as $\varepsilon\to0.$ This is the desingularization of a single vortex. However, it is not clear whether the vorticity $\omega^\varepsilon=\frac{1}{\varepsilon^2}(u^\varepsilon-\mu^\varepsilon)_+^p$ is an energy maximizer over the rearrangement class $\mathcal R({\omega^\varepsilon})$.
Our aim in this paper is to modify the vorticity method to obtain steady vortex flows with an energy characterization, moreover, the corresponding stream function satisfies a semilinear elliptic problem with a given profile function.

Now we turn to the precise statement of our main results. For technical reasons we need to impose some conditions on the profile function.
Let $f:\mathbb R\to \mathbb R$ be a function. We make the following assumptions on $f$.
\begin{itemize}
\item[(H1)] $f$ is continuous, $f(s)=0$ for $s\leq 0$, and $f$ is strictly increasing in $[0,+\infty)$.
\item[(H2)] There exists $\delta_0\in(0,1)$ such that
\[\int_0^sf(r)dr\leq \delta_0f(s)s,\,\,\forall\,s\geq0.\]
\item[(H3)] For all $\tau>0,$
\[\lim_{s\to+\infty}f(s)e^{-\tau s}=0.\]
\end{itemize}

Note that assumption (H2) implies $\lim_{s\to +\infty}f(s)=+\infty$. By using the identity $\int_0^sf(r)dr+\int_0^{f(s)}f^{-1}(r)dr=sf(s)$ for all $s\geq0,$ one can easily check that (H2) is in fact equivalent to
\begin{itemize}
 \item[ (H2)$'$] There exists $\delta_1\in(0,1)$ such that
\[F(s) \ge \delta_1 s f^{-1}(s),\,\,\forall\,s\geq0,\]
where $f^{-1}$ is defined as the inverse function of $f$ in $[0,+\infty)$ and $f^{-1}\equiv0$ in $(-\infty,0]$, and $F(s)=\int_0^sf^{-1}(r)dr$.
\end{itemize}

Note that many profile functions that frequently appear in nonlinear elliptic equations satisfy (H1)--(H3), for example $f(s)=s_+^p$ with $p\in(0,+\infty)$.

Our first result is about the desingularization of a single point vortex.
\begin{theorem}\label{single}
Let $f$ be a real function satisfying (H1)--(H3) and $\kappa$ be a fixed positive number. Then there exists $\varepsilon_0>0$ such that for any $\varepsilon\in(0,\varepsilon_0),$ there exists a solution $\omega^\varepsilon$ to \eqref{sv} having the form
\[\omega^\varepsilon=\frac{1}{\varepsilon^2}f(\mathcal{G}\omega^\varepsilon-\mu^\varepsilon), \,\,\int_D\omega^\varepsilon(x) dx=\kappa,\]
where $\mu^\varepsilon$ is a real number depending on $\varepsilon$ satisfying $\mu^\varepsilon=-\frac{\kappa}{2\pi}\ln\varepsilon+O(1)$ as $\varepsilon\to0^+$,
and the support of $\omega^\varepsilon$ shrinks to some point $\bar{x}\in D$, which is a global minimum point of the Robin function, that is,
\[\mbox{supp}(\omega^\varepsilon)\subset B_{o(1)}(\bar{x})\]
as $\varepsilon$ goes to zero. Moreover, $\omega^\varepsilon$ is a maximizer of the kinetic energy over $\mathcal R({\omega^\varepsilon})$.
\end{theorem}
Note that $H(x)\to+\infty$ as $x\to\partial D$, so $H$ attains its global minimum value in $D$.
\begin{remark}
For $f(s)=s_+,$ Theorem \ref{single} in fact provides a family of solutions to the plasma problem, which has been studied extensively in the literature. See Caffarelli--Friedman \cite{CF0}, Cao--Peng--Yan \cite{CPY0} for example.
\end{remark}

Our strategy of proving Theorem \ref{single} is as follows. We modify Turkington's method by considering the maximization of the following functional
\[\mathcal{E}(\omega)=E(\omega)-\mathcal{F}_\varepsilon(\omega),\,\,\mathcal{F}_\varepsilon (\omega)=\frac{1}{\varepsilon^2}\int_DF\big(\varepsilon^2\omega(x)\big)dx\]
with $F(s)=\int_0^sf^{-1}(r)dr$ over the following admissible class
\[\mathcal{A}_{\varepsilon,\Lambda}=\{\omega\in L^\infty(D)\mid 0\leq\omega\leq \frac{\Lambda}{\varepsilon^2}\,\mbox{ a.e. in }D, \int_D\omega(x)dx=\kappa\}.\]
It is not hard to prove that $\mathcal E$ attains its maximum value over $\mathcal{A}_{\varepsilon,\Lambda}$ and any maximizer satisfies
\begin{equation*}
\omega^{\varepsilon,\Lambda}=\frac{1}{\varepsilon^2}
f\big(\mathcal{G}\omega^{\varepsilon,\Lambda}-\mu^{\varepsilon,\Lambda}\big)
{\chi}_{_{\{x\in D\mid0<\mathcal{G}\omega^{\varepsilon,\Lambda}(x)-\mu^{\varepsilon,\Lambda}<f^{-1}(\Lambda)\}}}+\frac{\Lambda}{\varepsilon^2}\chi_{_{\{x\in D\mid\mathcal{G}\omega^{\varepsilon,\Lambda}(x)-\mu^{\varepsilon,\Lambda} \ge f^{-1}(\Lambda)\}}}
\end{equation*}
 for some $\mu^{\varepsilon,\Lambda}$ depending on $\varepsilon $ and $\Lambda$.
Then by analyzing the limiting behavior of the maximizer as $\varepsilon\to0$ we will show that if $\Lambda$ is sufficiently large, which does not depend on $\varepsilon,$ the functional $\mathcal{F}_\varepsilon$ plays a dominant role so that the patch part $\{x\in D\mid \mathcal{G}\omega^{\varepsilon,\Lambda}(x)-\mu^{\varepsilon,\Lambda}\ge f^{-1}(\Lambda)\}$ is empty. Moreover, by analyzing the energy like what Turkington did in \cite{T} we can show that the support of $\omega^{\varepsilon,\Lambda}$ ``shrinks" to a global minimum point of the Robin function. In our method, the parameter $\Lambda$ is new. Intuitively, as $\Lambda$ getting larger, the functional $\mathcal{F}_\varepsilon$ becomes more dominant relative to the quadratic term $E$, and finally completely eliminate the patch part.

As mentioned before, our construction also gives characterization of the energy of the solutions, which is essential to prove nonlinear stability. To make it clear, we recall the stability criterion proved by Burton \cite{B5}, which in our setting can be stated as follows.
\begin{theoremC}[Burton, \cite{B5}]
Let $\bar{\omega}\in L^\infty(D)$ be a steady solution to the vorticity equation. Suppose $\bar{\omega}\in L^\infty(D)$ is an isolated maximizer of the kinetic energy $E$ over $\mathcal{R}(\bar{\omega})$ in $L^p$ norm with $p\in[1,+\infty)$, that is, there exists $\delta_0>0$ such that for any $\omega\in\mathcal{R}(\bar{\omega})$, $0<\|\omega-\bar{\omega}\|_{L^p(D)}<\delta_0,$ we have $E(\bar{\omega})>E(\omega)$. Then $\bar{\omega}$ is nonlinearly stable in the following sense: for any $\epsilon>0,$ there exists $\delta>0$, such that for any initial vorticity $\omega_0\in \mathcal{R}(\bar{\omega})$ satisfying $\|\omega_0-\bar{\omega}\|_{L^p(D)}<\delta$, then the evolved vorticity $\omega(\cdot,t)$ of the Euler equation with initial vorticity $\omega_0$ satisfies $\|\omega(\cdot,t)-\bar{\omega}\|_{L^p(D)}<\epsilon$ for all $t\geq0$.
\end{theoremC}

By Burton's result, we are able to reduce nonlinear stability of $\omega^\varepsilon$ in Theorem \ref{single} to the uniqueness of an elliptic problem.
\begin{theorem}\label{stability}
Suppose that the $\bar{x}$ in Theorem \ref{single} is an isolated minimum point of the Robin function.
Suppose also that for sufficiently small $\varepsilon$ the solution to the following elliptic problem is unique
\begin{equation}\label{svef}
\begin{cases}
-\Delta u^\varepsilon=\frac{1}{\varepsilon^2}f(u^\varepsilon-\mu^\varepsilon),&x\in D,\\
u^\varepsilon=0,&x\in\partial D,\\
\int_D\frac{1}{\varepsilon^2}f(u^\varepsilon-\mu^\varepsilon)dx=1,\\
supp((u^\varepsilon-\mu^\varepsilon)_+)\subset B_{o(1)}(\bar{x}).
\end{cases}
\end{equation}
Then $\omega^\varepsilon$ is nonlinearly stable.
\end{theorem}
Note that by Caffarelli--Friedman \cite{CF}, if $D$ is a convex domain, then $H$ is a strictly convex function, thus $H$ has a unique(thus isolated) minimum point in $D$.

Our third result deals with steady Euler flows with vorticity that is sharply concentrated in a finite number of regions of small diameter.
Let $(\bar{x}_1,\cdot\cdot\cdot,\bar{x}_k)$ be an isolated minimum point of $\mathcal{W}_k$ (defined by \eqref{krf} with $\kappa_1,\cdot\cdot\cdot,\kappa_k$ be $k$ nonzero numbers) with $x_i\in D, i=1,\cdot\cdot\cdot,k$ and $x_i\neq x_j$ if $i\neq j$. For convenience we choose a small positive number $r_0$ such that $\overline{B_{r_0}(\bar{x}_i)}\subset\subset D$, $\overline{B_{r_0}(\bar{x}_i)}\cap \overline{B_{r_0}(\bar{x}_j)}=\varnothing$ if $i\neq j$, and
$(\bar{x}_1,\cdot\cdot\cdot,\bar{x}_k)$ is the unique minimum point of $\mathcal{W}_k$ in $\overline{B_{r_0}(\bar{x}_1)}\times\cdot\cdot\cdot\times\overline{B_{r_0}(\bar{x}_k)}$.

\begin{theorem}\label{multiple}
Let $f_1,\cdot\cdot\cdot,f_k$ be $k$ real functions satisfying (H1)--(H3). Then there exists a positive number $\varepsilon_0$ such that for any $\varepsilon\in(0,\varepsilon_0)$, there exists a solution to \eqref{sv} having the form
\[\omega^\varepsilon=\sum_{i=1}^k\omega^\varepsilon_i,\,\,
\omega^\varepsilon_i=\frac{1}{\varepsilon^2} sgn(\kappa_i)\chi_{B_{r_0}(\bar{x}_i)}f_i\big(\mbox{sgn}(\kappa_i)\mathcal{G}\omega^\varepsilon-\mu^\varepsilon_i\big), \,\,\int_D\omega^\varepsilon_idx=\kappa_i,\]
where $\mu^\varepsilon_i$ is a real number depending on $\varepsilon$ satisfying $\mu^\varepsilon_i=-\frac{|\kappa_i|}{2\pi}\ln\varepsilon+O(1)$ as $\varepsilon\to0^+$.
Moreover, the support of each $\omega^\varepsilon_i$ shrinks to $\bar{x}_i$, that is,
\[\mbox{supp}(\omega^\varepsilon_i)\subset B_{o(1)}(\bar{x}_i)\]
as $\varepsilon$ goes to zero.

\end{theorem}

The proof is by modifying the admissible $\mathcal{A}_{\varepsilon,\Lambda}$ by adding some suitable constraints on the support of the vorticity.

It is also worth mentioning that except for the desingularization type there is another type of steady Euler flows, which we call perturbation type. It consists of constructing steady Euler flows near a given one (usually a given nontrivial irrotational flow). The vorticity method and the stream function method still work in this situation. See Cao--Wang--Zhan \cite{CW3}, Li--Yang--Yan \cite{LYY}, Li--Peng \cite{LP} and the references therein. Finally, we bring to the attention of the reader that there is a similar situation with desingularization of vortex rings and shallow water vortices. See, e.g., Cao--Wan--Zhan \cite{CWZ}, Dekeyser \cite{D1, D2}, de Valeriola--Schaftingen \cite{DV} and Turkington \cite{Tur89}.

This paper is organized as follows. In Section 2, we deal with desingularization of a single point vortex by considering the maximization problem of $\mathcal E$ over $\mathcal{A}_{\varepsilon,\Lambda}$ and analyzing the limiting behavior of the maximizer as $\varepsilon\to0^+.$ We also give the proof of Theorem \ref{stability} in Section 2. In Section 3, we give the proof of Theorem \ref{multiple}.


\section{Proofs of Theorem \ref{single} and Theorem \ref{stability}}
In this section, we give the proofs of Theorem \ref{single} and Theorem \ref{stability}. As mentioned in Section 1, we first consider a maximization problem for the vorticity.

\subsection{Variational problem}
Let $\kappa>0$ be fixed and $\varepsilon>0$ be a parameter. Define
\begin{equation*}
\mathcal{A}_{\varepsilon,\Lambda}:=\{\omega\in L^\infty(D)~|~ 0\le \omega \le \frac{\Lambda}{\varepsilon^2}~ \mbox{ a.e. in }D, \int_{D}\omega(x)dx=\kappa \},
\end{equation*}
where $\Lambda>1$ is a sufficiently large real number such that $\mathcal{A}_{\varepsilon,\Lambda}$ is not empty. For example, we can take $\Lambda>\max\{1,\varepsilon^2\kappa|D|^{-1}\}$.
Consider the maximization problem of the following functional over $\mathcal{A}_{\varepsilon,\Lambda}$
$$\mathcal{E}(\omega)=\frac{1}{2}\int_D \omega(x)\mathcal{G}\omega(x)dx-\frac{1}{\varepsilon^2}\int_D F(\varepsilon^2\omega(x))dx,\,\,\omega\in \mathcal{A}_{\varepsilon,\Lambda}.$$
As mentioned in Section 1, we denote
\[E(\omega)=\frac{1}{2}\int_D \omega(x)\mathcal{G}\omega(x)dx,\,\, \mathcal F_\varepsilon(\omega)=\frac{1}{\varepsilon^2}\int_D F(\varepsilon^2\omega(x))dx,\,\,\,\omega\in \mathcal{A}_{\varepsilon,\Lambda}.\]
Since $F$ is a convex function, we can easily check that $\mathcal F_\varepsilon$ is a convex functional over $\mathcal{A}_{\varepsilon,\Lambda}$.
\begin{lemma}\label{lem1}
$\mathcal{E}$ is bounded from above and attains its maximum value over $\mathcal{A}_{\varepsilon,\Lambda}$.
\end{lemma}

\begin{proof}
Since $G(\cdot,\cdot)\in L^1(D\times D)$ we have
\[E(\omega)\leq \frac{\Lambda^2}{2\varepsilon^4}\int_D\int_D|G(x,y)|dxdy<+\infty,\,\,\forall\,\omega\in \mathcal{A}_{\varepsilon,\Lambda}.\]
For $\mathcal{F}_\varepsilon$ we have
\[|\mathcal{F}_\varepsilon|\leq \frac{1}{\varepsilon^2}F(\Lambda)|D|,\,\,\forall\,\omega\in \mathcal{A}_{\varepsilon,\Lambda}.\]
Therefore $\mathcal{E}$ is bounded from above over $\mathcal{A}_{\varepsilon,\Lambda}$.

Now let $\{\omega_{j}\}\subset \mathcal{A}_{\varepsilon,\Lambda}$ be a sequence such that as $j\to +\infty$
\[\mathcal{E}(\omega_{j}) \to \sup_{\omega\in \mathcal{A}_{\varepsilon,\Lambda}}\mathcal{E}({\omega}).\]
Since $\mathcal{A}_{\varepsilon,\Lambda}$ is a sequentially compact subset of $L^\infty(D)$ in the weak star topology(see \cite{CW1} for example), we may assume, up to a subsequence, that $\omega_j\to\bar{\omega}$ weakly star in $L^\infty(D)$ as $j\to\infty$ for some $\bar{\omega}\in \mathcal{A}_{\varepsilon,\Lambda}$.

Now we show that $\bar{\omega}$ is in fact a maximizer of $\mathcal{E}$ over $\mathcal{A}_{\varepsilon,\Lambda}$. To this end, it suffices to prove
\[\mathcal{E}(\bar{\omega})\geq\limsup_{j\to\infty}\mathcal{E}(\omega_j).\]
First by elliptic regularity theory we have $\mathcal{G}\omega_{j}\to \mathcal{G}\bar{\omega}$ in $C^1(\overline{D})$, from which we deduce that
\begin{equation}\label{Evar}
\lim_{j\to\infty}E(\omega_j)=E(\bar{\omega}).
\end{equation}
On the other hand, for $\mathcal{F}_\varepsilon$ we have
\begin{equation}\label{Fvar}
    \liminf_{j\to +\infty} \mathcal{F}_\varepsilon(\omega_j)\ge \mathcal{F}_\varepsilon(\bar\omega).
\end{equation}
In fact, we can prove \eqref{Fvar} by contradiction. Suppose that $\liminf_{j\to +\infty} \mathcal{F}_\varepsilon(\omega_j)\leq \mathcal{F}_\varepsilon(\bar\omega)+2\delta$ for some $\delta>0$. We may take a subsequence, still denoted by $\{\omega_j\}$, such that $\mathcal{F}_\varepsilon(\omega_j)\leq \mathcal{F}_\varepsilon(\bar\omega)+\delta$ for each $j$.
Since $\omega_j\to\bar{\omega}$ weakly star in $L^\infty(D)$ as $j\to\infty$, we have $\omega_j\to\bar{\omega}$ weakly in $L^2(D)$ as $j\to\infty$. By Mazur's theorem, we can take a sequence $\{w_n\}$ that converges to $\bar{\omega}$ strongly in $L^2(D)$, where each $w_n$ is made up of convex combinations of the $\omega_j$'s, that is, \[w_n=\sum_{j=1}^{m_n}\theta_{n_j}\omega_j,\,\,\,\sum_{j=1}^{m_n}\theta_{n_j}=1,\,\,\theta_{n_j}\in[0,1].\]
Without loss of generality we also assume that $w_n$ converges to $\bar{w}$ a.e. in $D$. Then by Lebesgue's dominated convergence theorem we obtain
\begin{equation}\label{aaa}
\lim_{n\to+\infty}\mathcal{F}_\varepsilon(w_n)=\mathcal{F}_\varepsilon(\bar{\omega}).
\end{equation}
On the other hand,
\[\mathcal{F}_\varepsilon(w_n)=\mathcal{F}_\varepsilon(\sum_{j=1}^{m_n}\theta_{n_j}\omega_j)\leq
\sum_{j=1}^{m_n}\theta_{n_j}\mathcal{F}_\varepsilon(\omega_j)\leq \sum_{j=1}^{m_n}\theta_{n_j}(\mathcal{F}_\varepsilon(\bar\omega)+\delta)=\mathcal{F}_\varepsilon(\bar\omega)+\delta,\]
which contradicts \eqref{aaa}. Here we used the convexity of $\mathcal F_\varepsilon$. Thus we have proved \eqref{Fvar}.
Combining \eqref{Evar} and \eqref{Fvar} we get the desired result.
\end{proof}

\begin{lemma}\label{lem1-1}
Let $\omega^{\varepsilon,\Lambda}$ be a maximizer of $\mathcal{E}$ over $\mathcal{A}_{\varepsilon,\Lambda}$. Then there exists some $\mu^{\varepsilon,\Lambda}$ such that
\begin{equation}\label{2-3}
\omega^{\varepsilon,\Lambda}=\frac{1}{\varepsilon^2}f(\psi^{\varepsilon,\Lambda}){\chi}_{\{x\in D\mid0<\psi^{\varepsilon,\Lambda}(x)<f^{-1}(\Lambda)\}}+\frac{\Lambda}{\varepsilon^2}\chi_{_{\{x\in D\mid\psi^{\varepsilon,\Lambda}(x) \geq f^{-1}(\Lambda)\}}} \ \ \mbox{ a.e. in }  D,
\end{equation}
where
\begin{equation}\label{2-4}
 \psi^{\varepsilon,\Lambda}:=\mathcal{G}\omega^{\varepsilon,\Lambda}-\mu^{\varepsilon,\Lambda}.
\end{equation}
Moreover, $\mu^{\varepsilon,\Lambda}$ has the following lower bound
\begin{equation}\label{2-5}
  \mu^{\varepsilon,\Lambda} \ge -f^{-1}(\Lambda).
\end{equation}

\end{lemma}
\begin{proof}
We take a family of test functions as follows
\begin{equation*}
  \omega_{s}=\omega^{\varepsilon,\Lambda}+s({\omega}-\omega^{\varepsilon,\Lambda}),\ \ \ s\in[0,1],
\end{equation*}
where ${\omega}$ is an arbitrary element of $\mathcal{A}_{\varepsilon,\Lambda}$. Since $\omega^{\varepsilon,\Lambda}$ is a maximizer, we have
\begin{equation*}
     0  \ge \frac{d\mathcal E(\omega_{s})}{ds}\bigg|_{s=0^+}
       =\int_{D}({\omega}-\omega^{\varepsilon,\Lambda})\big(\mathcal{G}\omega^{\varepsilon,\Lambda}-f^{-1}(\varepsilon^2\omega^{\varepsilon,\Lambda}))dx,
 \end{equation*}
that is,
\begin{equation*}
  \int_{D}\omega^{\varepsilon,\Lambda}\big(\mathcal{G}\omega^{\varepsilon,\Lambda}-f^{-1}(\varepsilon^2\omega^{\varepsilon,\Lambda}))dx\ge \int_{D}{\omega}\big(\mathcal{G}\omega^{\varepsilon,\Lambda}-f^{-1}(\varepsilon^2\omega^{\varepsilon,\Lambda}))dx
\end{equation*}
for all ${\omega}\in \mathcal{A}_{\varepsilon,\Lambda}.$
By an adaptation of the bathtub principle (see Lieb--Loss \cite{LL}, \S 1.14) we obtain
\begin{equation}\label{3-6}
  \begin{split}
    \mathcal{G}\omega^{\varepsilon,\Lambda}-\mu^{\varepsilon,\Lambda} &\ge f^{-1}(\varepsilon^2\omega^{\varepsilon,\Lambda}) \ \ \  \mbox{whenever}\  \omega^{\varepsilon,\Lambda}=\frac{\Lambda}{\varepsilon^2}, \\
     \mathcal{G}\omega^{\varepsilon,\Lambda}-\mu^{\varepsilon,\Lambda} &= f^{-1}(\varepsilon^2\omega^{\varepsilon,\Lambda})\ \ \  \mbox{whenever}\  0<\omega^{\varepsilon,\Lambda}<\frac{\Lambda}{\varepsilon^2}, \\
    \mathcal{G}\omega^{\varepsilon,\Lambda}-\mu^{\varepsilon,\Lambda} &\le f^{-1}(\varepsilon^2\omega^{\varepsilon,\Lambda}) \ \ \  \mbox{whenever}\  \omega^{\varepsilon,\Lambda}=0,
  \end{split}
\end{equation}
where $\mu^{\varepsilon,\Lambda}$ is a real number determined by
 $$\mu^{\varepsilon,\Lambda}=\inf\{s\in\mathbb R\mid|\{x\in D\mid\mathcal{G}\omega^{\varepsilon,\Lambda}-f^{-1}(\varepsilon^2\omega^{\varepsilon,\Lambda})>s\}|\le \frac{\kappa\varepsilon^2}{\Lambda}\}.$$
 Now the desired form $\eqref{2-3}$ follows immediately.

Next we prove \eqref{2-5}. We may suppose $\mu^{\varepsilon,\Lambda}<0$ (if otherwise, \eqref{2-5} holds true automatically). In this case, by \eqref{2-3} we have $\{x\in D\mid\omega^{\varepsilon,\Lambda}>0\}=D$ and $\omega^{\varepsilon,\Lambda}\ge\varepsilon^{-2} \min\{f(|\mu^{\varepsilon,\Lambda}|), \Lambda\}$ a.e.\,in $D$. Since $\int_D \omega^{\varepsilon,\Lambda}(x)dx =\kappa$, we conclude that
\begin{equation*}
  \min\{f(|\mu^{\varepsilon,\Lambda}|), \Lambda\}\le \frac{\varepsilon^2\kappa}{|D|}< \Lambda,
\end{equation*}
which clearly implies \eqref{2-5} by the strict monotonicity of $f$. Thus the proof is completed.
\end{proof}
\subsection{Limiting behavior}
In the following we analyze the limiting behavior of $\omega^{\varepsilon,\Lambda}$ as $\varepsilon\to 0^+.$  For convenience we will use $C$ to denote generic positive constants not depending on $\varepsilon$ and $\Lambda$ that may change from line to line.

We begin by giving a lower bound of $\mathcal{E}(\omega^{\varepsilon,\Lambda})$.
\begin{lemma}\label{lem2}
  $\mathcal{E}(\omega^{\varepsilon,\Lambda})\ge \frac{\kappa^2}{4\pi}\ln{\frac{1}{\varepsilon}}-C.$
\end{lemma}

\begin{proof}
The idea is to choose a suitable test function.
Let $x_0\in D$ be a fixed point.
Define \[\tilde{\omega}^{\varepsilon,\Lambda}=\frac{1}{\varepsilon^2} \chi_{_{B_{\varepsilon \sqrt{\kappa/\pi}}(x_0)}}.\] It is obvious that $\tilde{\omega}^{\varepsilon,\Lambda} \in \mathcal{A}_{\varepsilon,\Lambda}$ if $\varepsilon$ is sufficiently small. Therefore
\[ \mathcal{E}(\omega^{\varepsilon,\Lambda})\ge \mathcal{E}(\tilde{\omega}^{\varepsilon,\Lambda}).\]
By a simple calculation, we get
\begin{equation*}
  \mathcal{E}(\tilde{\omega}^{\varepsilon,\Lambda})\ge \frac{\kappa^2}{4\pi}\ln{\frac{1}{\varepsilon}}-C,
\end{equation*}
where the positive number $C$ does not depend on $\varepsilon$ and $\Lambda$. Thus the proof is completed.
\end{proof}

We now turn to estimate the Lagrange multiplier $\mu^{\varepsilon,\Lambda}$.
\begin{lemma}\label{lem3}
$ \mu^{\varepsilon,\Lambda}\ge \frac{\kappa}{2\pi}\ln{\frac{1}{\varepsilon}}-|1-2\delta_1|f^{-1}(\Lambda)-C,$ where $\delta_1$ is the positive number in (H2)$'$.
\end{lemma}

\begin{proof}
Recalling \eqref{2-3} and the assumption (H2)$'$ on $f$, we have
\begin{equation}\label{2-7}
\begin{split}
    2\mathcal E(\omega^{\varepsilon,\Lambda}) &=  \int_D\omega^{\varepsilon,\Lambda}\mathcal{G}\omega^{\varepsilon,\Lambda}dx-\frac{2}{\varepsilon^2}\int_DF(\varepsilon^2\omega^{\varepsilon,\Lambda})dx \\
     & \le\int_D \omega^{\varepsilon,\Lambda} \psi^{\varepsilon,\Lambda}dx-2\delta_1\int_D \omega^{\varepsilon,\Lambda} f^{-1}(\varepsilon^2\omega^{\varepsilon,\Lambda})dx+\kappa\mu^{\varepsilon,\Lambda} \\
     & \le |1-2\delta_1|\kappa f^{-1}(\Lambda)+\int_D\omega^{\varepsilon,\Lambda}\left(\psi^{\varepsilon,\Lambda}-f^{-1}(\Lambda)\right)_+dx +\kappa\mu^{\varepsilon,\Lambda}.
\end{split}
\end{equation}
Denote $U^{\varepsilon,\Lambda}:=\left(\psi^{\varepsilon,\Lambda}-f^{-1}(\Lambda)\right)_+$. Since $\mu^{\varepsilon,\Lambda}\geq-f^{-1}(\Lambda)$, we have $U^{\varepsilon,\Lambda}=0$ on $\partial D$.
So by integration by parts we have
\begin{equation}\label{eeeee}
 \int_D{|\nabla U^{\varepsilon,\Lambda}|^2}dx= \int_D \omega^{\varepsilon,\Lambda}U^{\varepsilon,\Lambda}dx.
\end{equation}
On the other hand, by H\"older's inequality and Sobolev's inequality
\begin{equation}\label{fffff}
\begin{split}
  \int_D \omega^{\varepsilon,\Lambda}U^{\varepsilon,\Lambda}dx
     & \le \frac{\Lambda}{\varepsilon^2}|\{x\in D\mid\omega^{\varepsilon,\Lambda}={\Lambda}{\varepsilon^{-2}}\}|^{\frac{1}{2}}\left(\int_D |U^{\varepsilon,\Lambda}|^2dx\right)^{\frac{1}{2}}\\
     & \le \frac{C\Lambda }{\varepsilon^2}|\{x\in D\mid\omega^{\varepsilon,\Lambda}={\Lambda}{\varepsilon^{-2}}\}|^{\frac{1}{2}}\int_D |\nabla U^{\varepsilon,\Lambda}|dx \\
     & \le C\left(\int_D {|\nabla U^{\varepsilon,\Lambda}|^2 dx}\right)^{\frac{1}{2}}.
\end{split}
\end{equation}
Here the positive constant $C$  does not depend on $\varepsilon$ and $\Lambda$. From \eqref{eeeee} and \eqref{fffff} we conclude that $\int_D \omega^{\varepsilon,\Lambda}U^{\varepsilon,\Lambda}dx$ is uniformly bounded with respect to $\varepsilon$ and $\Lambda$, which together with \eqref{2-7} and Lemma \ref{lem2} leads to the desired result.
\end{proof}

The following lemma shows that  $\psi^{\varepsilon,\Lambda}$ has a prior upper bound with respect to $\Lambda$.
\begin{lemma}\label{lem5}
$ \psi^{\varepsilon,\Lambda} \le |1-2\delta_1|f^{-1}(\Lambda)+\frac{\kappa}{4\pi} \ln \Lambda +C.$
\end{lemma}
\begin{proof}
For any $x\in D$, we have
\begin{equation*}
\begin{split}
   \psi^{\varepsilon,\Lambda}(x) &  \le \frac{1}{2\pi}\int_D\ln\frac{1}{|x-y|}\omega{^{\varepsilon,\Lambda}}(y)dy-\mu^{_{\varepsilon,\Lambda}}+C \\
     & \le \frac{\Lambda}{2\pi\varepsilon^2}\int_{B_{\varepsilon \sqrt{\kappa/(\Lambda\pi)}}(0)}\ln\frac{1}{|y|}dy-\mu^{_{\varepsilon,\Lambda}}+C \\
     & \le \frac{\kappa}{2\pi}\ln\frac{1}{\varepsilon}+\frac{\kappa}{4\pi}\ln\Lambda-\mu^{_{\varepsilon,\Lambda}}+C. \\
\end{split}
\end{equation*}
Hence by Lemma \ref{lem3} we have
\begin{equation*}
  \psi^{\varepsilon,\Lambda}(x) \le|1-2\delta_1|f^{-1}(\Lambda)+ \frac{\kappa}{4\pi} \ln \Lambda +C.
\end{equation*}
Since $x\in D$ is arbitrary, we conclude the proof.
\end{proof}

As a consequence of Lemma \ref{lem5}, we can eliminate the patch part in \eqref{2-3}.
\begin{lemma}\label{lem6}
If $\Lambda$ is sufficiently large(not depending on $\varepsilon$), then for all $\varepsilon>0$ we have
\begin{equation}
\label{mea0}|\{x\in D\mid\omega^{\varepsilon,\Lambda}(x)={\Lambda}{\varepsilon^{-2}}\}|=0.
\end{equation}
 As a consequence, $\omega^{\varepsilon,\Lambda}$ has the form
\begin{equation*}
  \omega^{\varepsilon,\Lambda}=\frac{1}{\varepsilon^2}f(\psi^{\varepsilon,\Lambda}).
\end{equation*}
\end{lemma}

\begin{proof}
  Notice that
\begin{equation}\label{2-9}
  \psi^{\varepsilon,\Lambda}\ge f^{-1}(\Lambda)\ \ \text{on}\ \  \{x\in D\mid\omega^{\varepsilon,\Lambda}(x)={\Lambda}{\varepsilon^{-2}}\}.
\end{equation}
Combining \eqref{2-9} and Lemma \ref{lem5}, we conclude that there exists some $C$ not depending on $\varepsilon$ and $\Lambda$ such that
\begin{equation}\label{2-10}
(1-|1-2\delta_1|) f^{-1}(\Lambda)\le \frac{\kappa}{4\pi}\ln \Lambda+C    \ \ \text{on}\ \  \{x\in D\mid\omega^{\varepsilon,\Lambda}(x)={\Lambda}{\varepsilon^{-2}}\}.
\end{equation}
Note that since $\delta_1\in(0,1),$ there holds $1-|1-2\delta_1|\in(0,1)$. Recall the assumption (H3) on $f$, that is, for each $\tau>0$
\begin{equation}
  \lim_{s\to+\infty}f(s)e^{-\tau s}=0,
\end{equation}
which implies for each $\tau>0$
\begin{equation}\label{2-11}
  \lim_{s\to+\infty}\left(\tau f^{-1}(s)-\ln s\right)=+\infty.
\end{equation}
Combining \eqref{2-10} and \eqref{2-11}, we deduce that if $\Lambda$ is sufficiently large, which does not depend on $\varepsilon$, then
 \[|\{x\in D\mid\omega^{\varepsilon,\Lambda}(x)={\Lambda}{\varepsilon^{-2}}\}|=0\]
 as desired.
\end{proof}

\begin{remark}
  Note that \eqref{2-11} is the only place we used (H3). Actually, it is easy to see that (H3) can be replaced by
   \begin{itemize}
   \item[(H3)$'$] There exists some $\tau_0>0$, which depends on $\delta_0$ and $\kappa$, such that
   \[ \lim_{s\to+\infty}f(s)e^{-\tau_0 s}=0.\]
   \end{itemize}
\end{remark}

In the rest of this section, we fix the parameter $\Lambda$ such that \eqref{mea0} holds. To simplify notations, we shall abbreviate $(\mathcal{A}_{\varepsilon,\Lambda},\omega^{\varepsilon,\Lambda},\psi^{\varepsilon,\Lambda}, \mu^{\varepsilon,\Lambda})$ as $(\mathcal{A}_{\varepsilon},\omega^{\varepsilon}, \psi^{\varepsilon}, \mu^{\varepsilon})$.

Now we turn to estimate the size and location of the supports of $\omega^\varepsilon$ as $\varepsilon\to0.$ To this end, we first give a general lemma that is used frequently in such problems.
\begin{lemma}\label{le7}
  Let $\Omega\subset D$, $0<\epsilon<1$, $A\ge0$, and let non-negative $\Gamma\in L^1(D)$, $\int_D \Gamma(x)dx=1$ and $||\Gamma||_{L^p(D)}\le C_1 \epsilon^{-2(1-{1}/{p})}$ for some $1<p\le +\infty$ and $C_1>0$. Suppose for any $x\in \Omega$, there holds
  \begin{equation}\label{3-12}
    (1-A)\ln\frac{1}{\epsilon}\le \int_D \ln\frac{1}{|x-y|}\Gamma(y)dy+C_2,
  \end{equation}
  where $C_2$ is a positive constant.
Then there exists some constant $R>1$ such that
\begin{equation*}
  diam(\Omega)\le R\epsilon^{1-2A}.
\end{equation*}
The constant $R$ may depend on $C_1$, $C_2$, but not on $A$, $\epsilon$.
\end{lemma}

\begin{proof}
We follow the strategy in Turkington \cite{T}. Let $q$ be the conjugate exponent of $p$, that is, $p^{-1}+q^{-1}=1$. Let $R_1>1$ be a fixed number. By \eqref{3-12}, for any $x\in \Omega$, we have
\begin{equation}\label{3-13}
  \begin{split}
     -A\ln\frac{1}{\epsilon} & \le \int_{B_{R_1\epsilon}(x)}\Big(\ln\frac{\epsilon}{|x-y|}\Big)_+\Gamma(y)dy+\int_{D \backslash B_{R_1\epsilon}(x)}\Big(\ln\frac{\epsilon}{|x-y|}\Big)_+\Gamma(y)dy+C_2 \\
       & \le \epsilon^{-2/q}\|\Gamma\|_{L^p(D)} \|\ln{|x|}\|_{L^q(B_1(0))}+\ln\frac{1}{R_1}\int_{D \backslash B_{R_1\epsilon}(x)}\Gamma(y)dy+C_2\\
       & \le \ln\frac{1}{R_1}\int_{D \backslash B_{R_1\epsilon}(x)}\Gamma(y)dy+ C_1 \|\ln{|x|}\|_{L^q(B_1(0))}+C_2  \\
       & \le \ln\frac{1}{R_1}\int_{D \backslash B_{R_1\epsilon}(x)}\Gamma(y)dy+ C.
  \end{split}
\end{equation}
Taking $R_1=R_2\epsilon^{-2A}$ with $R_2>1$ to be determined. Then \eqref{3-13} yields to
\begin{equation*}
   -A\ln\frac{1}{\epsilon}\le (-\ln R_2-2A\ln\frac{1}{\epsilon})\int_{D \backslash B_{R_2\epsilon^{1-2A}}(x)}\Gamma(y)dy+C,
\end{equation*}
that is,
\begin{equation}\label{3-14}
  \int_{D \backslash B_{R_2\epsilon^{1-2A}}(x)}\Gamma(y)dy\le \frac{A\ln\frac{1}{\epsilon}+C}{2A\ln\frac{1}{\epsilon}+\ln R_2}.
\end{equation}
Now, we fix $R_2$ large enough such that $\ln R_2>2C$. It follows from \eqref{3-14} that
\begin{equation}\label{3-15}
  \int_{D \cap B_{R_2\epsilon^{1-2A}}(x)}\Gamma(y)dy>\frac{1}{2}.
\end{equation}
Hence the lemma is proved by taking $R=2R_2$. In fact, suppose not, then there exist $x_1, x_2\in \Omega$ such that $B_{R_2\epsilon^{1-2A}}(x_1)\cap B_{R_2\epsilon^{1-2A}}(x_2)=\varnothing$. By \eqref{3-15}, we have
\begin{equation*}
  1=\int_D \Gamma(y)dy \ge \int_{B_{R_2\epsilon^{1-2A}}(x_1)}\Gamma(y)dy+\int_{B_{R_2\epsilon^{1-2A}}(x_2)}\Gamma(y)dy>1,
\end{equation*}
which leads to a contradiction.
\end{proof}

As a consequence of Lemma \ref{le7}, we are able to show that the size of $supp(\omega^\varepsilon)$ is of order $\varepsilon.$
\begin{lemma}\label{lem8}
There exists some $R_0>1$ independent of $\varepsilon$ such that
\begin{equation}
  diam\left(supp(\omega^\varepsilon)\right)\le R_0\varepsilon.
\end{equation}
\end{lemma}
\begin{proof}
Note that for each $ x\in supp(\omega^\varepsilon)$ there holds
\begin{equation*}
  \mathcal{G}\omega^\varepsilon(x)\ge \mu^\varepsilon\ge \frac{\kappa}{2\pi}\ln\frac{1}{\varepsilon}-C.
\end{equation*}
Now the desired result follows from Lemma \ref{le7}.
\end{proof}

We proceed to determine the limiting location of $supp(\omega^\varepsilon)$. Define the center of $\omega^\varepsilon$ by
\begin{equation*}
  x^\varepsilon:=\frac{1}{\kappa}\int_D x\omega^\varepsilon(x)dx.
\end{equation*}
From now on for the remainder of the discussion we fix a sequence $\varepsilon=\varepsilon_j\to 0^+$ such that
\begin{equation}\label{3-12}
  x^\varepsilon\to x^*\in \overline{D} \ \ \text{as}\ \ \varepsilon=\varepsilon_j\to +\infty.
\end{equation}

\begin{lemma}\label{le9}
  Any $x^*$ as in $\eqref{3-12}$ satisfies
  \begin{equation*}
    H(x^*)=\min_{x\in D}H(x).
  \end{equation*}
\end{lemma}

\begin{proof}
  For any $x_0\in D$, we set $\tilde{\omega}^\varepsilon(\cdot)=\omega^\varepsilon(x^\varepsilon-x_0+\cdot)$. Since $\omega^\varepsilon$ is a maximizer, we have $E(\omega^\varepsilon)\ge E(\tilde{\omega}^\varepsilon)$. Notice that
  \begin{equation*}
    \begin{split}
       \int_D\int_D\ln\frac{1}{|x-y|}\omega^\varepsilon(x)\omega^\varepsilon(y)dxdy & = \int_D\int_D \ln\frac{1}{|x-y|}\tilde{\omega}^\varepsilon(x)\tilde{\omega}^\varepsilon(y)dxdy, \\
       \int_D F(\varepsilon^2 \omega^\varepsilon)dx  &=  \int_D F(\varepsilon^2 \tilde{\omega}^\varepsilon)dx.
    \end{split}
  \end{equation*}
  Hence we obtain
   \begin{equation*}
     \frac{1}{2}\int_D\int_D h(x,y)\omega^\varepsilon(x)\omega^\varepsilon(y)dxdy\le \frac{1}{2}\int_D \int_D h(x,y)\tilde{\omega}^\varepsilon(x)\tilde{\omega}^\varepsilon(y)dxdy.
   \end{equation*}
   By passing $\varepsilon\to 0^+$, we get $H(x^*)\le H(x_0)$ as desired.
\end{proof}

We now turn to study the asymptotic shape of $\omega^\varepsilon$ by scaling technique. To this end, let $\zeta^\varepsilon\in L^\infty\big(B_{R_0}(0)\big)$ be defined by
\begin{equation*}
  \zeta^\varepsilon(x)=\varepsilon^2\omega^\varepsilon(x^\varepsilon+\varepsilon x),
\end{equation*}
where $R_0$ is the one in Lemma \ref{lem8}. We denote by $g^\varepsilon$ the symmetric radially nonincreasing Lebesgue-rearrangement of $\zeta^\varepsilon$ centered at the origin. The following result determines the asymptotic nature of $\omega^\varepsilon$ in terms of its scaled version $\zeta^\varepsilon$.

\begin{lemma}\label{lem10}
Every accumulation point of the family $\{\zeta^{\varepsilon}\}_{\varepsilon>0}$ in the weak topology of $L^2\big(B_{R_0}(0)\big)$ must be a radially nonincreasing function.
\end{lemma}

\begin{proof}
  Up to a subsequence we may assume that $\zeta^\varepsilon\to \zeta^*$ and $g^\varepsilon\to g^*$ weakly in $L^2\big(B_{R_0}(0)\big)$ as $\varepsilon\to 0^+$. By Riesz's rearrangement inequality, we first have
  \begin{equation*}
    \int_{B_{R_0}(0)}\int_{B_{R_0}(0)}\ln\frac{1}{|x-y|}\zeta^\varepsilon(x)\zeta^\varepsilon(y)dxdy\le  \int_{B_{R_0}(0)}\int_{B_{R_0}(0)}\ln\frac{1}{|x-y|}g^\varepsilon(x)g^\varepsilon(y)dxdy.
  \end{equation*}
Thus
\begin{equation}\label{2-16}
  \int_{B_{R_0}(0)}\int_{B_{R_0}(0)}\ln\frac{1}{|x-y|}\zeta^*(x)\zeta^*(y)dxdy\le  \int_{B_{R_0}(0)}\int_{B_{R_0}(0)}\ln\frac{1}{|x-y|}g^*(x)g^*(y)dxdy.
\end{equation}
Let $\tilde{\omega}^\varepsilon$ be defined as
\[
\tilde{\omega}^\varepsilon(x)=\left\{
   \begin{array}{lll}
        \varepsilon^{-2}g^\varepsilon\big(\varepsilon^{-1}(x-x^\varepsilon)\big) &    \text{if} & x\in B_{R_0\varepsilon}(x^\varepsilon), \\
         0                  &    \text{if} & x\in D\backslash B_{R_0\varepsilon}(x^\varepsilon).
    \end{array}
   \right.
\]
A direct calculation then yields that as $\varepsilon\to 0^+$,
\begin{equation*}
      \mathcal{E}({\omega}^\varepsilon)= \frac{1}{4\pi}\int_{B_{R_0}(0)}\int_{B_{R_0}(0)}\ln\frac{1}{|x-y|}\zeta^\varepsilon(x)\zeta^\varepsilon(y)dxdy+\frac{\kappa^2}{4\pi}\ln\frac{1}{\varepsilon}-H(x^*)
      -\mathcal F_\varepsilon(\omega^\varepsilon)+o(1), \\
\end{equation*}
and
\begin{equation*}
      \mathcal{E}(\tilde{\omega}^\varepsilon)= \frac{1}{4\pi}\int_{B_{R_0}(0)}\int_{B_{R_0}(0)}\ln\frac{1}{|x-y|}g^\varepsilon(x)g^\varepsilon(y)dxdy+\frac{\kappa^2}{4\pi}\ln\frac{1}{\varepsilon}-H(x^*)
      -\mathcal F_\varepsilon(\bar \omega^\varepsilon)+o(1). \\
\end{equation*}
Recalling that $\mathcal{E}(\tilde{\omega}^\varepsilon)\le \mathcal{E}({\omega}^\varepsilon)$ and $\mathcal F_\varepsilon( \omega^\varepsilon)=\mathcal F_\varepsilon(\tilde \omega^\varepsilon)$, we conclude that
\begin{equation*}
  \int_{B_{R_0}(0)}\int_{B_{R_0}(0)}\ln\frac{1}{|x-y|}\zeta^*(x)\zeta^*(y)dxdy\ge  \int_{B_{R_0}(0)}\int_{B_{R_0}(0)}\ln\frac{1}{|x-y|}g^*(x)g^*(y)dxdy,
\end{equation*}
which together with $\eqref{2-16}$ yield to
\begin{equation*}
  \int_{B_{R_0}(0)}\int_{B_{R_0}(0)}\ln\frac{1}{|x-y|}\zeta^*(x)\zeta^*(y)dxdy= \int_{B_{R_0}(0)}\int_{B_{R_0}(0)}\ln\frac{1}{|x-y|}g^*(x)g^*(y)dxdy.
\end{equation*}
By Lemma 3.2 in Burchard--Guo \cite{BG}, we know that there exists a translation $\mathcal T$ of $\mathbb{R}^2$ such that $\mathcal T\zeta^*=g^*$. Note that
\begin{equation*}
  \int_{B_{R_0}(0)}x\zeta^*(x)dx=\int_{B_{R_0}(0)}xg^*(x)dx=0.
\end{equation*}
Thus $\zeta^*=g^*$, the proof is completed.
\end{proof}

Now, we turn to study the limiting behavior of the corresponding stream functions $\psi^\varepsilon$.
We define the scaled versions of $\psi^\varepsilon$ as follows
\begin{equation*}
    \Psi^\varepsilon(y):=\psi^\varepsilon(x^\varepsilon+\varepsilon y), \ \ y\in D^\varepsilon:=\{y\in \mathbb{R}^2~|~x^\varepsilon+\varepsilon y\in D\}.
\end{equation*}
Thus, we have
\begin{equation}\label{2-19}
  -\Delta \Psi^\varepsilon =f(\Psi^\varepsilon)=\zeta^\varepsilon\ \ \text{in}\ D^\varepsilon, \ \ \int_{D^\varepsilon}f(\Psi^\varepsilon)dx=\kappa.
\end{equation}
 Note that $\{x\in D\mid \Psi^\varepsilon(x)>0\}\subset B_{R_0}(0)$. As in \cite{SV}, we introduce the limiting profile $U^\kappa: \mathbb{R}^2\to \mathbb{R}$ defined as the unique radially symmetric solution of the problem
\begin{equation}\label{2-20}
  \begin{cases}
    -\Delta U^\kappa= f(U^\kappa),&x\in\mathbb R^2, \\
    \int_{\mathbb{R}^2}f(U^\kappa)dx=\kappa.&
  \end{cases}
\end{equation}

\begin{lemma}\label{lem11}
  As $\varepsilon\to 0^+$, we have $\Psi^\varepsilon\to U^\kappa$ in $C^{1,\alpha}_{\text{loc}}(\mathbb{R}^2)$.
\end{lemma}
\begin{proof}
Note that $(\zeta^\varepsilon)$ is bounded in $L^\infty(D^\varepsilon)$. Thus, by classical elliptic estimates, the sequence $(\Psi^\varepsilon)$ is bounded in $W^{2,p}_{\text{loc}}(D^\varepsilon)$ for every $1\le p<+\infty$. By the Sobolev embedding theorem, we may conclude $(\Psi^\varepsilon)$ is compact in $C^{1,\alpha}_{\text{loc}}(D^\varepsilon)$ for every $0<\alpha<1$. Up to a subsequence we may assume $\zeta^\varepsilon\to \zeta$ weakly-star in $L^\infty(D^\varepsilon)$ and $\Psi^\varepsilon \to \Psi$ in $C_{\text{loc}}^{1,\alpha}(D^\varepsilon)$. By virtue of \eqref{2-19}, we get
\begin{equation*}
  -\Delta \Psi =f(\Psi)=\zeta\ \ \text{in}\ \mathbb{R}^2,\ \ \  \int_{\mathbb{R}^2}f(\Psi)dx=\kappa.
\end{equation*}
In view of Lemma \ref{lem10}, we know that $\zeta$ is a radially nonincreasing function, and hence $\Psi$ is radial as well. Therefore, we have $\Psi\equiv U^\kappa$. Thus the proof is completed.
\end{proof}

From Lemma \ref{lem11}, we can improve Lemma \ref{lem10} as follows.
\begin{corollary}\label{lem12}
As $\varepsilon\to 0^+$, one has $\zeta^\varepsilon\to f(U^\kappa)$ weakly star in $L^\infty(\mathbb{R}^2)$.
\end{corollary}

We end this subsection with two asymptotic expansions.
\begin{lemma}\label{lem13}
  The following asymptotic expansions hold as $\varepsilon\to 0^+$:
\begin{align}
\label{2-17}  \mathcal{E}(\omega^\varepsilon) & =\frac{\kappa^2}{4\pi}\ln\frac{1}{\varepsilon}+O(1), \\
\label{2-18}  \mu^\varepsilon & =\frac{\kappa}{2\pi}\ln\frac{1}{\varepsilon}+O(1).
\end{align}
\end{lemma}

\begin{proof}
  We first prove \eqref{2-17}. In fact, using Riesz's rearrangement inequality and the bathtub principle, we can conclude that
  \begin{equation*}
    \int_D\int_D\frac{1}{|x-y|}\omega(x)\omega(y)dxdy \le \kappa^2\ln\frac{1}{\varepsilon}+C,\ \ \forall\,\omega\in \mathcal{A}_{\varepsilon}.
  \end{equation*}
  Thus we have
  \begin{equation*}
    \mathcal{E}(\omega^\varepsilon) \le \frac{\kappa^2}{4\pi}\ln\frac{1}{\varepsilon}+C.
  \end{equation*}
  Combining this and Lemma \ref{lem2}, we clearly get \eqref{2-17}. Note that
  \begin{equation*}
  \begin{split}
     2\mathcal{E}(\omega^\varepsilon) & =\int_D \omega^\varepsilon(x)\psi^\varepsilon(x)dx-\frac{1}{\varepsilon^2}\int_D F(\varepsilon^2\omega(x))dx+\kappa\mu^\varepsilon \\
       & =\kappa\mu^\varepsilon+O(1)
  \end{split}
  \end{equation*}
  which together with \eqref{2-17} leads to \eqref{2-18}. The proof is completed.
\end{proof}
\begin{remark}
  Using Lemmas \ref{lem11} and \ref{lem12}, one may obtain some finer asymptotic expansions.
\end{remark}
\subsection{Proofs of Theorem \ref{single} and \ref{stability}}
Now we are ready to give the proofs of Theorem \ref{single} and \ref{stability}
\begin{proof}[Proof of Theorem \ref{single}]
It follows from the above lemmas and Theorem B.
\end{proof}
\begin{proof}[Proof of Theorem \ref{stability}]
Let $\varepsilon$ be a fixed small number such that the solution to \eqref{svef} is unique. Since $L^p$ norm and $L^1$ norm are equivalent on $\mathcal{R}(\omega^\varepsilon)$ for fixed $\varepsilon,$ we will only consider the case $p=1.$

Since $\omega^\varepsilon$ is a maximizer of $\mathcal{E}$ over $\mathcal{A}_\varepsilon$, taking into account the fact  $\mathcal{R}(\omega^\varepsilon)\subset \mathcal{A}_\varepsilon$ we immediately deduce that $\omega^\varepsilon$ is a maximizer of $\mathcal{E}$ over $\mathcal{R}(\omega^\varepsilon)$. But $\mathcal{F}_\varepsilon$ is a constant on $\mathcal{R}(\omega^\varepsilon)$, therefore we deduce that $\omega^\varepsilon$ is in fact a maximizer of ${E}$ over $\mathcal{R}(\omega^\varepsilon)$.
By Theorem C, to conclude the proof it suffices to show that $\omega^\varepsilon$ is an isolated maximizer of ${E}$ over $\mathcal{R}(\omega^\varepsilon)$.

 Let $\tilde{\omega}^\varepsilon$ be another maximizer of $E$ over $\mathcal{R}(\omega^\varepsilon)$ satisfying $\|\tilde{\omega}^\varepsilon-\omega^\varepsilon\|_{L^1(D)}<\kappa$. Since $\mathcal{F}_\varepsilon$ is a constant on $\mathcal{R}(\omega^\varepsilon)$, we deduce that $\mathcal{E}(\tilde{\omega}^\varepsilon)=\mathcal E(\omega^\varepsilon)$, which implies that $\tilde{\omega}^\varepsilon$ is in fact a maximizer of $\mathcal{E}$ over $\mathcal{A}_\varepsilon$. Then by the above discussion we see that $\tilde{u}^\varepsilon:=\mathcal G\bar{\omega}^\varepsilon$ satisfies
 \begin{equation}\label{svef2}
\begin{cases}
-\Delta \tilde{u}^\varepsilon=\frac{1}{\varepsilon^2}f(\tilde{u}^\varepsilon-\tilde \mu^\varepsilon),&x\in D,\\
\tilde{u}^\varepsilon=0,&x\in\partial D,\\
\int_D\frac{1}{\varepsilon^2}f(\tilde{u}^\varepsilon-\tilde \mu^\varepsilon)dx=1,\\
supp((\tilde{u}^\varepsilon-\tilde\mu^\varepsilon)_+)\subset B_{o(1)}(\tilde{x}).
\end{cases}
\end{equation}
Here $\tilde{x}$ is a global minimum point of the Robin function that may be different from $\bar{x}$. If $\bar{x}=\tilde{x}$, then by uniqueness we have $\bar{\omega}^\varepsilon={\omega}^\varepsilon$. If $\bar{x}\neq\tilde{x}$, then by the fact that $\bar{x}$ is an isolated minimum point of the Robin function, we deduce that $supp({\omega}^\varepsilon)\cap supp(\tilde{\omega}^\varepsilon)=\varnothing$ if $\varepsilon$ is small, from which we deduce that $\|\tilde{\omega}^\varepsilon-\omega^\varepsilon\|_{L^1(D)}=2\kappa$, which is a contradiction.
Therefore we have $\tilde{\omega}^\varepsilon=\omega^\varepsilon$, which completes the proof.

\end{proof}

\section{Proof of Theorem \ref{multiple}}
In this section we prove Theorem \ref{multiple}. Our idea is to solve a similar maximization problem as in Section 2 with some additional constraints on the support of the vorticity.

\subsection{Variational problem}
For any $\varepsilon>0$ and $\Lambda>\max\{1,\varepsilon^2|\kappa_1|/(\pi r_0^2),\cdot\cdot\cdot,\varepsilon^2|\kappa_k|/(\pi r_0^2)\}$, define
\begin{equation*}
\mathcal{A}^*_{\varepsilon,\Lambda}=\{\omega\in L^\infty(D)\mid\omega=\sum_{i=1}^k\omega_i, \mbox{supp($\omega_i$)}\subset B_{r_0}(\bar{x}_i), 0\le \mbox{sgn}(\kappa_i)\omega_i \le \frac{\Lambda}{\varepsilon^2}, \int_{D}\omega_idx=\kappa_i\}.
\end{equation*}
Consider the maximization problem of the following functional over $\mathcal{A}^*_{\varepsilon,\Lambda}$
$$\mathcal{E}(\omega)=\frac{1}{2}\int_D \omega(x)\mathcal{G}\omega(x)dx-\frac{1}{\varepsilon^2}\sum_{i=1}^k\int_D F_i(\varepsilon^2\mbox{sgn}(\kappa_i)\omega_i)dx,\,\,\omega_i=\omega\chi_{_{B_{r_0}(\bar{x}_i)}}.$$

\begin{lemma}\label{le1}
There exists $\bar{\omega}=\omega^{\varepsilon,\Lambda} \in \mathcal{A}^*_{\varepsilon,\Lambda}$ such that
\begin{equation}\label{3-2}
 \mathcal{E}(\omega^{\varepsilon,\Lambda})= \sup_{{\omega} \in  \mathcal{A}^*_{\varepsilon,\Lambda}}\mathcal{E}({\omega})<+\infty.
\end{equation}
Moreover,
there exist Lagrange multipliers $\mu^{\varepsilon,\Lambda}_{i}\in \mathbb{R}$, $i=1,\cdot\cdot\cdot,k,$ such that for each $i$
\begin{equation}\label{3-3}
      \mbox{sgn}(\kappa_i)\omega^{\varepsilon,\Lambda}_i=\frac{1}{\varepsilon^2}f_i(\psi^{\varepsilon,\Lambda}_{i})\chi_{_{B_{r_0}(\bar{x}_i)\cap\{x\in D\mid0<\psi^{\varepsilon,\Lambda}_{i}(x)<f^{-1}_i(\Lambda) \}}}+\frac{\Lambda}{\varepsilon^2}\chi_{_{B_{r_0}(\bar{x}_i)\cap\{x\in D\mid\psi^{\varepsilon,\Lambda}_{i}(x)\ge f^{-1}_i(\Lambda) \}}},
\end{equation}
where
\begin{equation}\label{3-4}
 \omega^{\varepsilon,\Lambda}_i=\omega^{\varepsilon,\Lambda}\chi_{B_{r_0}(\bar{x}_i)},\,\,\psi^{\varepsilon,\Lambda}_{i}=\mbox{sgn}(\kappa_i)\mathcal{G}\omega^{\varepsilon,\Lambda}-\mu^\varepsilon_{_{i, \Lambda}}.
\end{equation}
Moreover, each $\mu^{\varepsilon,\Lambda}_{i}$ has the following lower bound
\begin{equation}\label{3-5}
  \mu^{\varepsilon,\Lambda}_{i} \ge -f_i^{-1}({\Lambda})-C_0,
\end{equation}
where $C_0>0$ does not depend on $\varepsilon$ and $\Lambda$.
\end{lemma}

\begin{proof}
As in Lemma \ref{lem1} we may take a sequence $\omega^{j}\in  \mathcal{A}^*_{\varepsilon,\Lambda}$ such that as $j\to +\infty$
        \[\mathcal{E}(\omega^{j})  \to \sup_{{\omega} \in  \mathcal{A}^*_{\varepsilon,\Lambda}}\mathcal{E}(\bar{\omega}), \]
        \[\omega^{j}  \to \bar\omega \ \  \text{weakly star in $L^\infty(D)$}\]
for some $\bar\omega\in  \mathcal{A}^*_{\varepsilon,\Lambda}$. Since $\mathcal{G}\omega^{j}\to\mathcal{G}\bar\omega$ in $C^1(\overline{D})$ by elliptic regularity theory, we first have
\begin{equation*}
      \lim_{j\to +\infty}\int_D{\omega^{j}(x) \mathcal{G}\omega^{j}}(x)dx = \int_D{\bar\omega(x) \mathcal{G}\bar\omega}(x)dx\ \text{as}\ j\to +\infty.
\end{equation*}
On the other hand, we can also argue similarly as in Lemma \ref{lem11} to obtain
\begin{equation*}
    \liminf_{j\to +\infty} \frac{1}{\varepsilon^2}\int_D F_i(\varepsilon^2\mbox{sgn}(\kappa_i)\omega^j_i)\ge \frac{1}{\varepsilon^2}\int_D F_i(\varepsilon^2\mbox{sgn}(\kappa_i)\bar\omega_i),\ \ i=1,\cdot\cdot\cdot,k,
\end{equation*}
where we write $\omega^j_i=\omega^j\chi_{_{B_{r_0}(\bar{x}_i)}}$ and $\bar\omega_i=\bar\omega\chi_{_{B_{r_0}(\bar{x}_i)}}$. Consequently, we have \[\mathcal{E}(\bar\omega)=\lim_{j\to +\infty}\mathcal{E}(\omega^j)=\sup_{{\omega} \in  \mathcal{A}^*_{\varepsilon,\Lambda}}\mathcal{E}({\omega}).\]

We now show that each maximizer $\bar{\omega}$ must be of the form \eqref{3-3}.
Consider the following family of test functions
\begin{equation*}
  \omega_{s}=\bar\omega+s({\omega}-\bar\omega),\ \ \ s\in[0,1],
\end{equation*}
for arbitrary ${\omega}\in  \mathcal{A}^*_{\varepsilon,\Lambda}$. Since $\bar\omega$ is a maximizer, we have
\[0  \ge \frac{d\mathcal{E}(\bar\omega_{s})}{ds}\bigg|_{s=0^+}
        =\int_{D}({\omega}-\bar\omega)\left(\mathcal{G}\bar\omega-\sum_{i=1}^k\mbox{sgn}(\kappa_i)f_i^{-1}(\varepsilon^2\mbox{sgn}(\kappa_i)\bar\omega_i)\right)dx,
\]
that is, for any ${\omega}\in  \mathcal{A}^*_{\varepsilon,\Lambda}$
\begin{equation*}
\begin{split}
  &\int_{D}\bar\omega\left(\mathcal{G}\bar\omega-\sum_{i=1}^k\mbox{sgn}(\kappa_i)f_i^{-1}(\varepsilon^2\mbox{sgn}(\kappa_i)\bar\omega_i)\right)dx\\
  \ge& \int_{D}{\omega}\left(\mathcal{G}\bar\omega-\sum_{i=1}^k\mbox{sgn}(\kappa_i)f_i^{-1}(\varepsilon^2\mbox{sgn}(\kappa_i)\bar\omega_i)\right)dx.
\end{split}
\end{equation*}
By using an adaptation of the bathtub principle, we obtain
\begin{equation*}
\mbox{sgn}(\kappa_i)\bar\omega_i=\frac{1}{\varepsilon^2}f_i(\psi^{\varepsilon,\Lambda}_{i})
\chi_{_{B_{r_0}(\bar{x}_i)\cap\{x\in D\mid0<\psi^\varepsilon_{_{i, \Lambda}}(x)<f_i(\Lambda) \}}}+\frac{\Lambda}{\varepsilon^2}\chi_{_{B_{r_0}(\bar{x}_i)\cap\{x\in D\mid\psi^{\varepsilon,\Lambda}_{i}(x)\ge f_i(\Lambda) \}}}, \,\  i=1,\cdot\cdot\cdot,k,
\end{equation*}
where
$$\bar\omega_i=\bar\omega\chi_{_{B_{r_0}(\bar{x}_i)}},\,\,\psi^{\varepsilon,\Lambda}_{i}=\mbox{sgn}(\kappa_i)\mathcal{G}\bar\omega-\mu^{\varepsilon,\Lambda}_{i},$$
$$\mu^{\varepsilon,\Lambda}_{i}=\inf\{t:|\{x\in B_{r_0}(\bar{x}_i)\,|\, \psi^{\varepsilon,\Lambda}_{i}(x)>t\}|\le \frac{|\kappa_i|\varepsilon^2}{\Lambda}\}\in \mathbb{R}.$$ Now the stated form $\eqref{3-3}$ follows immediately. Finally, we prove \eqref{3-5}. Notice that $\psi^{\varepsilon,\Lambda}_{i}\ge-C_0$. By virtue of $\int_D \bar\omega_i dx =\kappa_i$, we conclude that
\begin{equation*}
  \min\{f_i(-\mu^{\varepsilon,\Lambda}_{i}-C_0), \Lambda\}\le \frac{|\kappa_i|\varepsilon^2}{\pi r^2_0}<\Lambda.
\end{equation*}
This clearly implies \eqref{3-5}. The proof is thus completed.
\end{proof}

\subsection{Limiting behavior and proof of Theorem \ref{multiple}}
As in Section 2, we analyze the limiting behavior of $\omega^{\varepsilon,\Lambda}$ as $\varepsilon\to0^+.$ We will use $C$ to denote various positive number that does not depend on $\varepsilon$ and $\Lambda$ in this subsection. For the sake of convenience we also define
$$\mathcal{E}_i(\omega)=\frac{1}{2}\int_D \omega_i(x)\mathcal{G}\omega_i(x)dx-\frac{1}{\varepsilon^2}\int_DF_i\left(\varepsilon^2\mbox{sgn}(\kappa_i)\omega_i(x)\right)dx,\ \ i=1,\cdot\cdot\cdot,k.$$
It is not hard to check that as $\varepsilon\to0^+$
\begin{equation}\label{simple}
\mathcal{E}(\omega)=\sum_{i=1}^{k}\mathcal{E}_i(\omega)+O(1),\ \ \forall\,\omega\in  \mathcal{A}^*_{\varepsilon,\Lambda},
\end{equation}
where the bounded quantity $O(1)$ does not depend on $\varepsilon$, $\Lambda$ and $\omega$.

\begin{lemma}\label{le2}
We have the following lower bound for $\mathcal{E}(\omega^{\varepsilon,\Lambda})$
\begin{equation}\label{3-f}  \mathcal{E}(\omega^{\varepsilon,\Lambda})\ge \sum_{i=1}^{k}\frac{\kappa_i^2}{4\pi}\ln{\frac{1}{\varepsilon}}-C.
\end{equation}

\end{lemma}

\begin{proof}
We choose a test function $\tilde{\omega}^\varepsilon \in  \mathcal{A}^*_{\varepsilon,\Lambda}$ as follows $$\tilde{\omega}^\varepsilon=\sum_{1\le j\le k, j\not=i}^{k}\omega^{\varepsilon,\Lambda}_j     +\frac{\text{sgn}(\kappa_i)}{\varepsilon^2}\chi_{_{B_{\varepsilon\sqrt{|\kappa_i|/\pi}}(\bar{x}_i)}}.$$
Note that $\mathcal{E}(\omega^{\varepsilon,\Lambda})\ge \mathcal{E}(\tilde{\omega}^\varepsilon)$. By a simple calculation, we get \begin{equation}\label{3-6}
\mathcal{E}_i(\omega^{\varepsilon,\Lambda})\ge \frac{\kappa_i^2}{4\pi}\ln\frac{1}{\varepsilon}-C,\ \ i=1,\cdot\cdot\cdot,k,
\end{equation}
where the positive number $C$ does not depend on $\varepsilon$ and $\Lambda$.
 Combining \eqref{simple} and \eqref{3-6}, we get \eqref{3-f}. The proof is completed.
\end{proof}

We now turn to estimate the Lagrange multiplier $\mu^{\varepsilon,\Lambda}_{i} $.
\begin{lemma}\label{le3}
There exists $\varepsilon_1>0$ not depending on $\Lambda$, such that for every $\varepsilon\in (0,\varepsilon_1)$, we have
  \begin{equation}\label{3-7}
    \mu^{\varepsilon,\Lambda}_{i}\ge\frac{|\kappa_i|}{2\pi}\ln{\frac{1}{\varepsilon}} -|1-2\delta_1|f_i^{-1}(\Lambda)-C,\ \ i=1,\cdot\cdot\cdot,k.
  \end{equation}
\end{lemma}

\begin{proof}
Recalling \eqref{3-3} and the assumption (H2)$'$ on each $f_i$, we have
\begin{equation}\label{3-8}
\begin{split}
    2\mathcal{E}_i(\omega^{\varepsilon,\Lambda}) &=  \int_D\omega^{\varepsilon,\Lambda}_{i}\mathcal{G}\omega^{\varepsilon,\Lambda}_{i}dx
    -\frac{2}{\varepsilon^2}\int_DF_i(\varepsilon^2\text{sgn}(\kappa_i)\omega^{\varepsilon,\Lambda}_{i})dx \\
     & \le\int_D |\omega^{\varepsilon,\Lambda}_{i}| \psi^{\varepsilon,\Lambda}_{i}dx-2\delta_1\int_D |\omega^{\varepsilon,\Lambda}_{i}| f_i^{-1}(\varepsilon^2\text{sgn}(\kappa_i)\omega^{\varepsilon,\Lambda}_{i})dx+\mu^{\varepsilon,\Lambda}_{i}|\kappa_i|+C \\
     & \le |1-2\delta_1||\kappa_i|f_i^{-1}(\Lambda)+\int_D|\omega^{\varepsilon,\Lambda}_{i}|\left(\psi^{\varepsilon,\Lambda}_{i}-f_i^{-1}(\Lambda)\right)_+dx
     +\mu^{\varepsilon,\Lambda}_{i}|\kappa_i|+C.
\end{split}
\end{equation}
Let $C_0$ be as in Lemma \ref{le1}. Set
\begin{equation*}
  U^{\varepsilon,\Lambda}_{i}:=\left(\psi^{\varepsilon,\Lambda}_{i}-f_i^{-1}(\Lambda)\right)_+,\ \ \tilde{U} ^{\varepsilon,\Lambda}_{i}:=\left(\mathcal{G}|\omega^{\varepsilon,\Lambda}_{i}|-\mu^{\varepsilon,\Lambda}_{i}-f_i^{-1}(\Lambda)-C_0\right)_+.
\end{equation*}
Note that $-\Delta \mathcal{G}|\omega^{\varepsilon,\Lambda}_{i}|=|\omega^{\varepsilon,\Lambda}_{i}|$. We multiply $\tilde{U} ^{\varepsilon,\Lambda}_{i}$ on both sides of this equation and integrate by parts to get
\begin{equation}\label{3-9}
\begin{split}
   &\,\,\,\,\,\,\,\int_D{|\nabla \tilde{U}^{\varepsilon,\Lambda}_{i}|^2}dx\\
   &= \int_D |\omega^{\varepsilon,\Lambda}_{i}|\tilde{U}^{\varepsilon,\Lambda}_{i}dx \\
     & \le \int_D |\omega^{\varepsilon,\Lambda}_{i}|{U}^{\varepsilon,\Lambda}_{i}dx+C\\
     & \le \frac{\Lambda}{\varepsilon^2}|\{|\omega^{\varepsilon,\Lambda}_{i}(x)|={\Lambda}{\varepsilon^{-2}}\}|^{\frac{1}{2}}\left(\int_{B_{r_0}(\bar{x}_i)} |U^{\varepsilon,\Lambda}_{i}|^2dx\right)^{\frac{1}{2}}+C\\
     & \le \frac{C\Lambda }{\varepsilon^2}|\{|\omega^{\varepsilon,\Lambda}_{i}(x)|={\Lambda}{\varepsilon^{-2}}\}|^{\frac{1}{2}}\int_{B_{r_0}(\bar{x}_i)} (|\nabla U^{\varepsilon,\Lambda}_{i}|+|U^{\varepsilon,\Lambda}_{i}| )dx +C\\
     & \le \frac{C\Lambda }{\varepsilon^2}|\{\omega^{\varepsilon,\Lambda}_{i}(x)|={\Lambda}{\varepsilon^{-2}}\}|^{\frac{1}{2}}\int_{\{|\omega^{\varepsilon,\Lambda}_{i}(x)|={\Lambda}{\varepsilon^{-2}}\}}(|\nabla \tilde{U}^{\varepsilon,\Lambda}_{i}|+|\tilde{U}^{\varepsilon,\Lambda}_{i}|)dx +C\\
     & \le C|\kappa_i|\left(\int_D {|\nabla \tilde{U}^{\varepsilon,\Lambda}_{i}|^2 }dx\right)^{\frac{1}{2}}+C|\{|\omega^{\varepsilon,\Lambda}_{i}(x)|={\Lambda}{\varepsilon^{-2}}\}|^{\frac{1}{2}}\int_D |\omega^{\varepsilon,\Lambda}_{i}|\tilde{U}^{\varepsilon,\Lambda}_{i}dx+C\\
     & \le C|\kappa_i|\left(\int_D {|\nabla \tilde{U}^{\varepsilon,\Lambda}_{i}|^2 }dx\right)^{\frac{1}{2}}+C\varepsilon \sqrt{|\kappa_i|}\int_D |\omega^{\varepsilon,\Lambda}_{i}|\tilde{U}^{\varepsilon,\Lambda}_{i}dx+C,
\end{split}
\end{equation}
where we used H\"older's inequality and Sobolev embedding $W^{1,1}(B_{r_0}(\bar{x}_i))\hookrightarrow L^2(B_{r_0}(\bar{x}_i))$, and the positive constant $C$  does not depend on $\varepsilon$ and $\Lambda$. From \eqref{3-9}, we conclude that if $\varepsilon<1/(2C{|\kappa_i|^{1/2}})$, then $\int_D |\omega^{\varepsilon,\Lambda}_{i}|U^{\varepsilon,\Lambda}_{i}dx$ is uniformly bounded with respect to $\varepsilon$, $\Lambda$. Now \eqref{3-7} clearly follows from \eqref{3-6} and \eqref{3-8}. The proof is completed.
\end{proof}

The following lemma shows that $\psi^{\varepsilon,\Lambda}_{i}$ has a prior upper bound with respect to $\Lambda$.
\begin{lemma}\label{le5}
Let $\varepsilon_1$ be as in Lemma \ref{le3}. Then for every $\varepsilon\in (0,\varepsilon_1)$, we have
\begin{equation*}
  \psi^{\varepsilon,\Lambda}_{i}(x) \le |1-2\delta_1|f_i^{-1}(\Lambda)+\frac{|\kappa_i|}{4\pi} \ln \Lambda +C,\ \ \forall\,x\in B_{r_0}(\bar{x}_i),\ \ i=1,\cdot\cdot\cdot,k.
\end{equation*}
\end{lemma}
\begin{proof}
For any $x\in B_{r_0}(\bar{x}_i)$, we have
\begin{equation*}
\begin{split}
   \psi^{\varepsilon,\Lambda}_{i}(x) &  \le \frac{1}{2\pi}\int_D\ln\frac{1}{|x-y|}|\omega^{\varepsilon,\Lambda}_{i}|(y)dy-\mu^{\varepsilon,\Lambda}_{i}+C \\
     & \le \frac{\Lambda}{2\pi\varepsilon^2}\int_{B_{{\varepsilon\sqrt{|\kappa_i|/{\Lambda\pi}}}}(0)}\ln\frac{1}{|y|}dy-\mu^{\varepsilon,\Lambda}_{i}+C \\
     & \le \frac{|\kappa_i|}{2\pi}\ln\frac{1}{\varepsilon}+\frac{|\kappa_i|}{4\pi}\ln\Lambda-\mu^{\varepsilon,\Lambda}_{i}+C \\
\end{split}
\end{equation*}
Hence, by Lemma \ref{le3}, we have
\begin{equation*}
  \psi^{\varepsilon,\Lambda}_{i}(x) \le|1-2\delta_1|f_i^{-1}(\Lambda)+ \frac{|\kappa_i|}{4\pi} \ln \Lambda +C.
\end{equation*}
The proof is completed.
\end{proof}

Using Lemma \ref{le5}, we can further deduce the following result.
\begin{lemma}\label{le6}
Let $\varepsilon_1$ be as in Lemma \ref{le3}. If $\Lambda>1$ is sufficiently large, which does not depend on $\varepsilon$, then for every $\varepsilon\in (0,\varepsilon_1)$, we have
\begin{equation*}
  \omega^{\varepsilon,\Lambda}_{i}=\frac{\text{sgn}(\kappa_i)}{\varepsilon^2}f_i\big(\psi^{\varepsilon,\Lambda}_{i}\big)\chi_{_{B_{r_0}(\bar{x}_i)}},\ \ \ i=1,\cdot\cdot\cdot,k.
\end{equation*}
\end{lemma}

\begin{proof}
  Notice that
\begin{equation*}
  \psi^{\varepsilon,\Lambda}_{i}\ge f_i^{-1}(\Lambda)\ \ \text{on}\ \  \{x\in D\mid|\omega^{\varepsilon,\Lambda}_{i}(x)|={\Lambda}{\varepsilon^{-2}}\}.
\end{equation*}
Combining this with Lemma \ref{le5}, we conclude that for some $C$ independent of $\varepsilon$ and $\Lambda$,
\begin{equation}\label{3-10}
(1-|1- 2\delta_1|)f_i^{-1}(\Lambda)\le \frac{|\kappa_i|}{4\pi}\ln \Lambda+C\ \ \text{on}\ \  \{x\in D\mid|\omega^{\varepsilon,\Lambda}_{i}(x)|={\Lambda}{\varepsilon^{-2}}\}.
\end{equation}
But the assumption (H3) on $f_i$ implies for each $\tau>0$
\begin{equation}\label{3-11}
  \lim_{s\to +\infty}(\tau f_i^{-1}(s)-\ln s)=+\infty.
\end{equation}
Combining \eqref{3-10} and \eqref{3-11}, we deduce that
  $$|\{x\in D\mid |\omega^{\varepsilon,\Lambda}_{i}(x)|={\Lambda}{\varepsilon^{-2}}\}|=0$$
  if $\Lambda$ is sufficiently large and $0<\varepsilon<\varepsilon_1$. The proof is thus completed.
\end{proof}

In the sequel $\Lambda$ is assumed to be fixed and large enough such that the conclusion in Lemma \ref{le6} holds true, and $C$ will be used to denote various positive numbers not depending on $\varepsilon$. To simplify notation we shall abbreviate $(\mathcal{A}^*_{\varepsilon,\Lambda},\omega^{\varepsilon,\Lambda}, \omega^{\varepsilon,\Lambda}_{i},\psi^{\varepsilon,\Lambda}_{i}, \mu^{\varepsilon,\Lambda}_{i})$ as $(\mathcal{A}^*_{\varepsilon},\omega^{\varepsilon},\omega^{\varepsilon}_i, \psi^{\varepsilon}_i, \mu^{\varepsilon}_i)$.

Note that for every $\varepsilon\in (0,\varepsilon_1)$, there holds
\begin{equation*}
  \mathcal{G}|\omega_i^{\varepsilon}|(x)\ge \psi^\varepsilon_i(x)-C\ge \mu^\varepsilon_i-C,\ \ \forall\, x\in \text{supp
  }(\omega^\varepsilon_i),\ \ i=1,\cdot\cdot\cdot,k,
\end{equation*}
 Combining this and Lemma \ref{le7}, we get
\begin{lemma}\label{le8}
There exists some $R_0>1$ independent of $\varepsilon$ such that
\begin{equation}
  diam(supp(\omega^{\varepsilon}_i))\le R_0\varepsilon,\ \ i=1,\cdots,k.
\end{equation}
\end{lemma}

With the estimates of $\text{supp}(\omega^{\varepsilon}_i)$, we can further determine the location of $\text{supp}(\omega^{\varepsilon}_i)$. To this end, we define the center of each $\omega^{\varepsilon}_i$ to be
\begin{equation*}
  x^{\varepsilon}_i=\frac{1}{\kappa_i}\int_D x\omega^{\varepsilon}_i(x)dx.
\end{equation*}

\begin{lemma}\label{le9}
$
  \lim_{\varepsilon\to 0^+}(x^{\varepsilon}_1,\cdots,x^{\varepsilon}_k)=(\bar{x}_1,\cdots,\bar{x}_k).
$
\end{lemma}

\begin{proof}
Up to a subsequence we may assume that $$\lim_{\varepsilon\to 0^+}(x^{\varepsilon}_1,\cdots,x^{\varepsilon}_k)=(x^*_1,\cdots,x^*_k)\in\overline{B_{r_0}(\bar{x}_1)}\times\cdots\times\overline{B_{r_0}(\bar{x}_k)}.$$
  For any $(x_1,\cdots,x_k)\in B_{r_0}(\bar{x}_1)\times\cdots\times B_{r_0}(\bar{x}_k)$, set
\begin{equation*}
  \tilde{\omega}^\varepsilon=\sum_{i=1}^{k}\tilde{\omega}^\varepsilon_i,\,\,\,\, \tilde{\omega}^\varepsilon_i(\cdot)=\omega^\varepsilon_i(x^\varepsilon_i-x_i+\cdot).
\end{equation*}
Since $\omega^\varepsilon$ is a maximizer, we have $E(\omega^\varepsilon)\ge E(\tilde{\omega}^\varepsilon)$. Observe that
  \begin{equation*}
    \begin{split}
       \int_D\int_D \ln\frac{1}{|x-y|}\omega^\varepsilon_i(x)\omega^\varepsilon_i(y)dxdy & = \int_D\int_D \ln\frac{1}{|x-y|}\tilde{\omega}^\varepsilon_i(x)\tilde{\omega}^\varepsilon_i(y)dxdy, \\
       \int_DF(\varepsilon^2 \omega^\varepsilon_i)dx &=  \int_DF(\varepsilon^2 \tilde{\omega}^\varepsilon_i)dx.
    \end{split}
  \end{equation*}
  Hence we obtain
   \begin{equation*}
   \begin{split}
       &\sum_{{i\not=j},{1\le i,j\le k}} \,   \int_D\int_D G(x,y)\omega^\varepsilon(x)_i\omega^\varepsilon_j(y)dxdy-\sum_{i=1}^{k}\int_D\int_Dh(x,y)\omega^\varepsilon_i(x)\omega^\varepsilon_i(y)dxdy   \\
        &\,\, \ge \sum_{{i\not=j},{1\le i,j\le k}}\,  \int_D\int_DG(x,y)\tilde{\omega}^\varepsilon_i(x)\tilde{\omega}^\varepsilon_j(y)dxdy-\sum_{i=1}^{k}\int_D\int_Dh(x,y)\tilde{\omega}^\varepsilon_i(x)\tilde{\omega}^\varepsilon_i(y)dxdy.
   \end{split}
   \end{equation*}
  Letting $\varepsilon\to 0^+$, we obtain
  \begin{equation*}
    \mathcal{W}_k(x^*_i,\cdot\cdot\cdot,x^*_k)\le \mathcal{W}_k(x_1,\cdot\cdot\cdot,x_k).
  \end{equation*}
  Since $(\bar{x}_1,\cdot\cdot\cdot,\bar{x}_k)$ is the unique minimum point of $\mathcal{W}_k$ in $\overline{B_{r_0}(\bar{x}_1)}\times\cdot\cdot\cdot\times\overline{B_{r_0}(\bar{x}_k)}$, we must have $(x^*_i,\cdot\cdot\cdot,x^*_k)=(\bar{x}_1,\cdot\cdot\cdot,\bar{x}_k)$. The proof is completed.
\end{proof}

Combining Lemmas \ref{le8} and \ref{le9}, we immediately get the following result.
\begin{lemma}\label{10}
  If  $\varepsilon>0$ is sufficiently small, we have
\begin{equation*}
  dist\big(supp(\omega^\varepsilon_i),\partial B_{r_0}(\bar{x}_i)\big)>0,\ \ \ i=1,\cdot\cdot\cdot,k.
\end{equation*}
\end{lemma}

We now turn to study the asymptotic shape of the optimal vortices. As before, let $\zeta^\varepsilon_i\in L^\infty\big(B_{R_0}(0)\big)$ be defined by
\begin{equation*}
  \zeta^\varepsilon_i(x)=sgn(\kappa_i)\varepsilon^2\omega^\varepsilon_i(x^\varepsilon_i+\varepsilon x),\ \ i=1,\cdot\cdot\cdot,k,
\end{equation*}
for $R_0>1$ as in Lemma \ref{le8}. We denote by $g^\varepsilon_i$ the symmetric radially nonincreasing Lebesgue-rearrangement of $\zeta^\varepsilon_i$ centered on $0$. The following result is a counterpart of Lemma \ref{lem10}, which determines the asymptotic nature of $\omega^\varepsilon_i$ in terms of its scaled version $\zeta^\varepsilon_i$.

\begin{lemma}\label{le10}
Let $i\in\{1,\cdots,k\}$. Then every accumulation point of the family $\{\zeta^{\varepsilon}_i:\varepsilon>0\}$ in the weak topology of $L^2\big(B_{R_0}(0)\big)$ must be a radially nonincreasing function.
\end{lemma}

\begin{proof}
  Up to a subsequence we may assume that $\zeta^\varepsilon_i\to \zeta^*_i$ and $g^\varepsilon_i\to g^*_i$ weakly in $L^2\big(B_{R_0}(0)\big)$ as $\varepsilon\to 0^+$. By Riesz's rearrangement inequality, we first have
  \begin{equation*}
    \int_{B_{R_0}(0)}\int_{B_{R_0}(0)}\ln\frac{1}{|x-y|}\zeta^\varepsilon_i(x)\zeta^\varepsilon_i(y)dxdy\le  \int_{B_{R_0}(0)}\int_{B_{R_0}(0)}\ln\frac{1}{|x-y|}g^\varepsilon_i(x)g^\varepsilon_i(y)dxdy.
  \end{equation*}
Thus
\begin{equation}\label{3-16}
  \int_{B_{R_0}(0)}\int_{B_{R_0}(0)}\ln\frac{1}{|x-y|}\zeta^*_i(x)\zeta^*_i(y)dxdy\le  \int_{B_{R_0}(0)}\int_{B_{R_0}(0)}\ln\frac{1}{|x-y|}g^*_i(x)g^*_i(y)dxdy.
\end{equation}
Let $\tilde{\omega}^\varepsilon_i$ be defined as
\[
\tilde{\omega}^\varepsilon_i(x)=\left\{
   \begin{array}{lll}
        \varepsilon^2g^\varepsilon_i\big(\varepsilon^{-1}(x-x^\varepsilon_i)\big) &    \text{if} & x\in B_{R_0\varepsilon}(x^\varepsilon_i), \\
         0                  &    \text{if} & x\in D\backslash B_{R_0\varepsilon}(x^\varepsilon_i).
    \end{array}
   \right.
\]
Let $\tilde{\omega}^\varepsilon=\sum_{1\le j\le k,j\not=i}^{k}\omega^\varepsilon_j+\tilde{\omega}^\varepsilon_i$.
A direct calculation then yields to,
\begin{equation*}
\begin{split}
      \mathcal{E}({\omega}^\varepsilon)= \frac{1}{4\pi}\int_{B_{R_0}(0)}\int_{B_{R_0}(0)}&\ln\frac{1}{|x-y|}\zeta^\varepsilon_i(x)\zeta^\varepsilon_i(y)dxdy+\sum_{j=1}^{k}\frac{\kappa_j^2}{4\pi}\ln\frac{1}{\varepsilon} \\
     &     +\sum_{j\not=i}^{k} \frac{1}{4\pi}\int_{B_{R_0}(0)}\int_{B_{R_0}(0)}\ln\frac{1}{|x-y|}\zeta^\varepsilon_j(x)\zeta^\varepsilon_j(y)dxdy+\mathcal{R}^\varepsilon_1,
\end{split}
\end{equation*}
and
\begin{equation*}
\begin{split}
      \mathcal{E}(\tilde{\omega}^\varepsilon)= \frac{1}{4\pi}\int_{B_{R_0}(0)}\int_{B_{R_0}(0)}&\ln\frac{1}{|x-y|}g^\varepsilon_i(x)g^\varepsilon_i(y)dxdy +\sum_{j=1}^{k}\frac{\kappa_j^2}{4\pi}\ln\frac{1}{\varepsilon}\\
     &     +\sum_{j\not=i}^{k} \frac{1}{4\pi}\int_{B_{R_0}(0)}\int_{B_{R_0}(0)}\ln\frac{1}{|x-y|}\zeta^\varepsilon_j(x)\zeta^\varepsilon_j(y)dxdy+\mathcal{R}^\varepsilon_2,
\end{split}
\end{equation*}
where
\begin{equation*}
  \lim_{\varepsilon\to 0^+}\mathcal{R}^\varepsilon_1=\lim_{\varepsilon\to 0^+}\mathcal{R}^\varepsilon_2\in \mathbb{R}.
\end{equation*}
Recalling that $\mathcal{E}(\tilde{\omega}^\varepsilon)\le \mathcal{E}({\omega}^\varepsilon)$, we conclude that
\begin{equation*}
  \int_{B_{R_0}(0)}\int_{B_{R_0}(0)}\ln\frac{1}{|x-y|}\zeta^*_i(x)\zeta^*_i(y)dxdy\ge  \int_{B_{R_0}(0)}\int_{B_{R_0}(0)}\ln\frac{1}{|x-y|}g^*_i(x)g^*_i(y)dxdy.
\end{equation*}
which together with $\eqref{3-16}$ yields to
\begin{equation*}
  \int_{B_{R_0}(0)}\int_{B_{R_0}(0)}\ln\frac{1}{|x-y|}\zeta^*_i(x)\zeta^*_i(y)dxdy= \int_{B_{R_0}(0)}\int_{B_{R_0}(0)}\ln\frac{1}{|x-y|}g^*_i(x)g^*_i(y)dxdy.
\end{equation*}
By Lemma 3.2 in Burchard--Guo \cite{BG}, we know that there exists a translation $\mathcal T$ in $\mathbb{R}^2$ such that $\mathcal T\zeta^*_i=g^*_i$. Taking into account
\begin{equation*}
  \int_{B_{R_0}(0)}x\zeta^*_i(x)dx=\int_{B_{R_0}(0)}xg^*_i(x)dx=0,
\end{equation*}
we obtain $\zeta^*_i=g^*_i$. The proof is thus completed.
\end{proof}

Now, we turn to study the limiting behavior of the corresponding stream functions $\psi_i^\varepsilon$.
We define the scaled versions of $\psi_i^\varepsilon$ as follows
\begin{equation*}
    \Psi_i^\varepsilon(y)=\psi_i^\varepsilon(x_i^\varepsilon+\varepsilon y), \ \ y\in D_i^\varepsilon:=\{y\in \mathbb{R}^2~|~x_i^\varepsilon+\varepsilon y\in D\}.
\end{equation*}
Thus, we have
\begin{equation}\label{3-30}
\begin{split}
   -\Delta \Psi_i^\varepsilon(y)  & = \zeta^\varepsilon_i(y)+\text{sgn}(\kappa_i)\sum_{1\le j\le k,j\not=i}^{k}\varepsilon^2\omega^\varepsilon_j(x^\varepsilon_i+\varepsilon y)\\
     & =f_i\big(\Psi^\varepsilon_i\big)\chi_{_{B_{r_0}(\bar{x}_i)}}(x^\varepsilon_i+\varepsilon  y)+\text{sgn}(\kappa_i)\sum_{1\le j\le k,j\not=i}^{k}\varepsilon^2\omega^\varepsilon_j(x^\varepsilon_i+\varepsilon y),\\
   \int_{D_i^\varepsilon}&f_i\big(\Psi^\varepsilon_i\big)dx =|\kappa_i|, \ \ i=1,\cdot\cdot\cdot,k.
\end{split}
\end{equation}
 Note that $\{\Psi^\varepsilon_i>0\}\subset B_{R_0}(0)$. As before, we now introduce the limiting function $U^{|\kappa_i|}: \mathbb{R}^2\to \mathbb{R}$ defined as the unique radially symmetric solution of the following elliptic problem
\begin{equation}\label{3-31}
  \begin{cases}
    -\Delta U^{|\kappa_i|}= f_i(U^{|\kappa_i|}),& \\
    \int_{\mathbb{R}^2}f_i(U^{|\kappa_i|})=|\kappa_i|.&
  \end{cases}
\end{equation}

Arguing as in the proof of Lemma \ref{lem11}, we can obtain the following result.
\begin{lemma}\label{le11}
  As $\varepsilon\to 0^+$, we have $\Psi_i^\varepsilon\to U^{|\kappa_i|}$ in $C^{1,\alpha}_{\text{loc}}(\mathbb{R}^2)$.
\end{lemma}
As before, we can now sharpen Lemma \ref{le11} as follows.
\begin{corollary}\label{le12}
  As $\varepsilon\to 0^+$, one has $\zeta_i^\varepsilon\to f_i(U^{|\kappa_i|})$ weakly star in $L^\infty(\mathbb{R}^2)$.
\end{corollary}

Arguing similarly as in the proof of Lemma \ref{lem13}, we can obtain the following expansions.
\begin{lemma}\label{le13}
  The following asymptotic expansions hold as $\varepsilon\to 0^+$
\begin{align}
\label{3-17}  \mathcal{E}_i(\omega^\varepsilon) & =\frac{\kappa_i^2}{4\pi}\ln\frac{1}{\varepsilon}+O(1),\ \ i=1,\cdot\cdot\cdot,k,\\
\label{3-18}  \mu^\varepsilon_i & =\frac{|\kappa_i|}{2\pi}\ln\frac{1}{\varepsilon}+O(1),\ \ i=1,\cdot\cdot\cdot,k,\\
\label{3-19} \mathcal{E}(\omega^\varepsilon)&= \sum_{i=1}^{k}\frac{\kappa_i^2}{4\pi}\ln\frac{1}{\varepsilon}+O(1).
\end{align}
\end{lemma}

\begin{proof}[Proof of Theorem \ref{multiple}]
It follows from the above lemmas and Theorem B.

\end{proof}


\phantom{s}
 \thispagestyle{empty}

\end{document}